\documentclass[11pt]{amsart}
\usepackage{amsfonts}
\usepackage{amsmath}
\usepackage{amssymb}
\usepackage{multicol}
\usepackage{stmaryrd}
\usepackage{cite}
\usepackage{epsfig}
\usepackage{color}
\usepackage{graphics}
\usepackage{graphicx}
\usepackage{multicol,graphics}
\newcommand\bes{\begin{eqnarray}}
\newcommand\ees{\end{eqnarray}}

\newtheorem{theorem}{Theorem}[section]
\newtheorem{lemma}[theorem]{Lemma}

\numberwithin{equation}{section}
\allowdisplaybreaks

\begin{document}

\title[A nonlocal diffusion model
with free boundaries]{A Lotka-Volterra competition model with
nonlocal diffusion and free boundaries}
\author[J.F. Cao, W.T. Li and J. Wang]{Jia-Feng Cao$^\dag$,\, Wan-Tong Li$^\ddag$ and Jie Wang$^\dag$}
\thanks{\hspace{-.6cm} $^\dag$Department of Applied Mathematics, Lanzhou University of Technology, Lanzhou, Gansu, 730050, P.R. China.
\\
$^\ddag$School of Mathematics and Statistics, Lanzhou University,
Lanzhou, Gansu, 730000, P.R. China.
\\
$^\S$({\sf Corresponding author~~wtli@lzu.edu.cn})}

\date{\today}

\maketitle

\begin{abstract}
This paper is concerned with a nonlocal diffusion Lotka-Volterra type competition model that consisting of a native species and an invasive species in a one-dimensional habitat with free boundaries. We prove the well-posedness of the system and get a spreading-vanishing dichotomy for the invasive species. We also provide some sufficient conditions to ensure spreading success or spreading failure for the case that the invasive species is an inferior competitor or a superior competitor, respectively.



\end{abstract}

\section{Introduction}
\noindent

In this paper we are interested in the dynamics of the solution $(u(t,x),v(t,x),g(t),h(t))$ which is governed by the following
nonlocal dispersal model with free boundaries in the one space
dimension
\begin{equation}
\left\{
\begin{aligned}
&u_t=d_1\left[\int_{g(t)}^{h(t)}J(x-y)u(t,y)dy-u\right]
+u(a_1-b_1u-c_1v),& &t>0,~x\in(g(t),h(t)),\\
&v_t=d_2\left[\int_{\mathbb{R}}J(x-y)v(t,y)dy-v\right]
+v(a_2-b_2u-c_2v),& &t>0,~x\in\mathbb{R},\\
&u(t,g(t))=u(t,h(t))=0,& &t>0,\\
&h'(t)=\mu\int_{g(t)}^{h(t)}\int_{h(t)}^{+\infty}J(x-y)u(t,x)dydx,& &t>0,\\
&g'(t)=-\mu\int_{g(t)}^{h(t)}\int_{-\infty}^{g(t)}J(x-y)u(t,x)dydx,& &t>0,\\
&v(0,x)=v_0(x),& &x\in\mathbb{R},\\
&u(0,x)=u_0(x),~h(0)=-g(0)=h_0,& &x\in[-h_0,h_0]
\end{aligned}
\right.
\label{101}
\end{equation}
where $u(t,x)$ represents the density of an invasive species  and $v(t,x)$ denotes the density of a native species. It is imposed that $u$ exists in the initial region $[-h_0,h_0]$ and extends into the habitat by the spreading fronts $x=g(t)$ and $x=h(t)$, which to be determined together with $u(t,x)$ and $v(t,x)$, see Cao et al. \cite{CDLL-2018} for the detailed derivation of the free boundary conditions $h'(t)$ and $g'(t)$. The constants $d_1$ and $d_2$ represent the dispersal rate of species $u$ and $v$, respectively; $\mu$ is a positive constant accounting for the expanding ability of $u$. The initial data satisfy
\begin{equation}
\left\{
\begin{aligned}
&v_0(x)\in C(\mathbb{R})\cap L^\infty(\mathbb{R}),~v_0(x)>0~\text{ in }~\mathbb{R},\\
&u_0(x)\in C([-h_0,h_0]),~u_0(-h_0)=u_0(h_0)=0\text{ and }
u_0(x)>0\text{ in }~(-h_0,h_0),
\end{aligned}
\right.
\label{102}
\end{equation}
and the kernel
$J: \mathbb{R}\rightarrow\mathbb{R}$ is a non-negative continuous
function on $\mathbb{R}$. More precisely, we assume in what follows that
\begin{description}
\item[(J)]$J(0)>0, \int_{\mathbb{R}}J(x)dx=1$, $J$ is symmetric and $\sup_{\mathbb{R}}J\le\infty$.
\end{description}

The problem (\ref{101}) is a nonlocal variation of a Lotka-Volterra model in which the diffusion is described by the Laplace operator, and the free boundary is denoted by the Stefan condition. For example, Du and Lin \cite{DYH2014} considered the long time behavior of the following problem
\begin{equation}
\left\{
\begin{aligned}
&u_t-d_1\Delta u=u(a_1-b_1u-c_1v), & &t>0,0\le x<h(t),\\
&v_t-d_2\Delta v=v(a_2-b_2u-c_2v), & &t>0, 0\le x<\infty,\\
&u_x(t,0)=v_x(t,0)=0,~u(t,x)=0,& &t>0,h(t)\le x<\infty,\\
&h'(t)=-\mu u_x(t,h(t)), & &t>0,\\
&u(0,x)=u_0(x),~h(0)=h_0, & &0\le x\le h_0,\\
&v(0,x)=v_0(x), & &0\le x<\infty.
\end{aligned}
\right.
\label{101-local}
\end{equation}
In the case that $u$ is a superior competitor in the sense that
$\frac{a_1}{a_2}>\max\left\{\frac{b_1}{b_2},~\frac{c_1}{c_2}\right\}$, they established a spreading vanishing dichotomy for species $u$,
that is, either $h(t)\rightarrow+\infty$ and $(u,v)\rightarrow\left(\frac{a_1}{b_1},0\right)$ as $t\rightarrow+\infty$, or $h(t)<+\infty$ and $(u,v)\rightarrow\left(0,\frac{a_2}{c_2}\right)$ as $t\rightarrow+\infty$; while in the case that $u$ is an inferior competitor, namely, $\frac{a_1}{a_2}<\min\left\{\frac{b_1}{b_2},~\frac{c_1}{c_2}\right\}$, they showed that
$(u,v)\rightarrow\left(0,\frac{a_2}{c_2}\right)$ as $t\rightarrow+\infty$, which means that the native species $v$ win the competition. When the spreading of
$u$ happens, they further showed a rough estimate for the spreading speed. These results have been extended to many Lotka-Volterra two species models, one can refer to Du et al. \cite{Du-Wang-Zhou-JMPA-2017}, Guo and Wu \cite{GJS-JDDE-2012,GJS-2015-Nonlinearity}, Wang \cite{wang2014}, Wang and Zhao \cite{WangMX-ZhaoJF-JDDE-2014} and references cited therein.

It is well-known that the invasion and spreading of nonlocal diffusion Lotka-Volterra type competition models have been studied intensively. In 2003, Hutson et al. \cite{Hutson2003JMB} studied the competitive advantages and disadvantages of diffusive rate and diffusive distance in a spatially heterogeneous environment, namely, the following problem
\begin{equation}
\left\{
\begin{aligned}
&u_t=d_1\left[\frac{1}{(L_u)^N}\int_{\Omega}J\left(\frac{x-y}{L_u}\right)u(t,y)dy-u(t,x)\right]+uf(u+v,x), & &(t, x)\in (0,\infty)\times\overline\Omega,\\
&v_t=d_2\left[\frac{1}{(L_v)^N}\int_{\Omega}J\left(\frac{x-y}{L_v}\right)v(t,y)dy-v(t,x)\right]+vf(u+v,x), & &(t, x)\in (0,\infty)\times\overline\Omega,
\end{aligned}
\right.
\label{101-nonlocal}
\end{equation}
where $d_i(i=1,2)$ and $J(\cdot)$ are the same as them in (\ref{101}); the constants $L_u, L_v>0$ characterize the diffusive distance (interpreted as spreads in\cite{Bates2006,Hutson2003JMB}). They showed that as in the case of reaction-diffusion models, for fixed spread slower rates of diffusion are always optimal, that is, for any non-trivial, non-negative initial conditions, if $L_u=L_v$ and $d_1<d_2$, then the semi-trivial equilibrium $(u^*,0)$ is globally asymptotically stable. While fixing the diffusion rate ($d_1=d_2$) and varying the spread,  in the case of small spread, the smaller spread is selected (the semi-trivial equilibrium in the presence of the species with the smaller spread is the global attractor) and in the case of large spread the larger spread is selected.

We also mention that traveling wave solutions of nonlocal diffusive competition systems have been studied intensively. See e.g. Bao et al. \cite{BaoXX-JDE-2016,BLW2015}, Pan et al. \cite{PanSX-ZAMP-2009}, Zhao and Ruan \cite{ZR2011} for the existence of traveling wave fronts, and Du et al. \cite{DLW2016, DLW2018,DLW2019}, Li et al. \cite{LWT-ZhangL-ZhangGB-DCDS-2015} and Wang and Lv \cite{WL2010} for the existence of invasion entire solutions. Also, there are many works concerned with the spectral theory of nonlocal dispersal operators and entire solutions of nonlocal dispersal equations, see Coville \cite{CovilleJDE2010}, Hetzer and Shen \cite{Hetzer-2013}, Li et al. \cite{LSW2010,LWW2008,LWZ2016}, Shen and Zhang \cite{SZ2010,Shen-Zhang-Proceeding-2012}, Sun et al. \cite{SLW2012,SLY2014}, Sun et al. \cite{SZLW2019}, Yang et al. \cite{YLS2016}, Zhang et al. \cite{ZLW2012} and Zhang et al. \cite{ZLW2017,ZLWS2019}.

The main purpose of this paper is extend the above results into the free boundary problem with nonlocal diffusion. It must be emphasized that our approach to deal with the nonlocal diffusion problem (\ref{101}) is totally different from these of the responding random (local) diffusion equations, including the well-posedness as well as the long-term behaviors. In particular, we establish a comparison principle (see Theorem \ref{Comparison-Principle-1}) in a suitable parabolic domain to consider the global asymptotic stability of the
semi-trivial equilibriums $\left(\frac{a_1}{b_1},0\right)$ and $\left(0,\frac
{a_2}{c_2}\right)$ with the different initial datas. Moreover, we give a precise classification of the dynamics for the invasion species $u$ in the case that $u$ is a superior competitor or an  inferior one.

This paper is organized as follows. Section 2 is concerned with the existence and uniqueness of positive solutions of (\ref{101}), which was established by two times using of the contraction mapping theorem. Section 3 is devoted to the dynamics of the solutions that obtained in Section 2. We first collect some essential results among the principal eigenvalues, and then establish a spreading vanishing dichotomy for species $u$ in the case that $u$ is a superior competitor. We further obtain a sharp criteria of expanding ability $\mu$ to ensure spreading or vanishing in the end.

In the end of this section, we must mention that after the completion of this article, we received the preparation paper of Du et al.\cite{Du-Wang-Zhao}. They considered the equation \eqref{101} with \textbf{the two species both located in the same growth domain} and obtained some interesting asymptotical behavior of solutions to \eqref{101} with general reaction term. We also would like to mention the article \cite{WangMX-WangJP}. Wang and Wang  considered a class of free boundary problems of ecological models with nonlocal and local diffusions that can be regarded as extensions of free boundary problems for reaction diffusion systems.

\section{The well-posedness of (\ref{101})}
\noindent

This section is focused on the global existence and uniqueness of solutions to the problem (\ref{101}). For convenient, we introduce some notations first.
For given $h_0, T>0$, we define
\begin{align*}
&\mathbb{H}_{h_0,T}=\Big\{h\in C([0,T])~\Big|~h(0)=h_0,
~~\inf_{0\le t_1<t_2\le T}\frac{h(t_2)-h(t_1)}{t_2-t_1}>0\Big\},\\
&\mathbb{G}_{h_0,T}=\Big\{g\in C([0,T])~\Big|-g(t)\in\mathbb{H}_{h_0,T}\Big\},\\
&C_0([-h_0,h_0])=\Big\{u\in C([0,T])~\Big|~u(-h_0)=u(h_0)=0\text{ and }u_0(x)>0\text{ in }(-h_0,h_0)\Big\}.
\end{align*}

For $g\in\mathbb{G}_{h_0,T}$, $h\in\mathbb{H}_{h_0,T}$
and $u_0(x)\in C_0([-h_0,h_0])$, define
\begin{align*}
&\Omega_{g,h}=\left\{(t,x)\in\mathbb{R}^2: 0<t\le T,
~g(t)<x<h(t)\right\},\\
&\Omega_{\infty}=\left\{(t,x)\in\mathbb{R}^2: 0<t\le T,~x\in\mathbb{R}\right\},\\
&\mathbb{X}_{v_0,\infty}=\Big\{\phi\in C(\Omega_{\infty})~\Big|
~\phi(0,x)=v_0(x)~\text{for}~x\in\mathbb{R}~\text{ and }~0<\phi\le M_1\text{ in }\Omega_{\infty}\Big\},\\
&\mathbb{X}_{u_0,g,h}=\Big\{\psi\in C(\bar\Omega_{g,h})~\Big|~\psi\ge0\text{ in }\Omega_{g,h},
~\psi(0,x)=u_0(x)~\text{for}~x\in[-h_0,h_0]\\
&\quad\quad\quad\quad\quad\quad\quad\quad\quad\quad\quad\quad
\quad\text{and}~\psi(t,g(t))=
\psi(t,h(t))=0~\text{for }~0\le t\le T\Big\}
\end{align*}
with $M_1:=\max\left\{a_1,K_0,\|v_0\|_{L^\infty}\right\}$.
Following is Maximum Principle that will be frequently used later.
\begin{lemma}{\rm(Maximum Principle\cite{CDLL-2018})}
Assume that {\rm\bf(J)} holds and for some given $h_0,T>0$, let $g\in\mathbb{G}_{h_0,T}$ and $h\in\mathbb{H}_{h_0,T}$. Assume
that for all $(t,x)\in\bar\Omega_{g,h}$, functions $u(t,x)$
and $u_t(t,x)$ are continuous, and there exists a function $c(t,x)\in L^\infty(\Omega_{g,h})$ such that
\begin{equation}
\left\{
\begin{aligned}
&u_t(t,x)\ge d\int_{g(t)}^{h(t)}J(x-y)u(t,y)dy-du+c(t,x)u, & &
(t,x)\in \Omega_{g,h},\\
&u(t,g(t))\ge0,~u(t,h(t))\ge0,& &t>0,\\
&u(0,x)\ge0,& &x\in[-h_0,h_0].
\end{aligned}
\right.
\label{201}
\end{equation}
Then $u(t,x)\ge0$ for all $0\le t\le T$ and $x\in[g(t),h(t)]$.
Further more, if $u(0,x)\not\equiv0$ in $[-h_0,h_0]$, then
$u(t,x)>0$ for all $(t,x)\in\Omega_{g,h}$.
\label{Maximum-Principle}
\end{lemma}

\begin{proof}
See the proof of \cite[Lemma 2.2]{CDLL-2018}.
\end{proof}

We first prove a existence and uniqueness result for a general free boundary problem. Consider the following free boundary problem
\begin{equation}
\left\{
\begin{aligned}
&u_t=d_1\left[\int_{g(t)}^{h(t)}J(x-y)u(t,y)dy-u(t,x)\right]
+f_1(u,v),& &t>0,~x\in(g(t),h(t)),\\
&v_t=d_2\left[\int_{\mathbb{R}}J(x-y)v(t,y)dy-v(t,x)\right]
+f_2(u,v),& &t>0,~x\in\mathbb{R},\\
&u(t,g(t))=u(t,h(t))=0,& &t>0,\\
&h'(t)=\mu\int_{g(t)}^{h(t)}\int_{h(t)}^{+\infty}J(x-y)u(t,x)dydx,& &t>0,\\
&g'(t)=-\mu\int_{g(t)}^{h(t)}\int_{-\infty}^{g(t)}J(x-y)u(t,x)dydx,& &t>0,\\
&v(0,x)=v_0(x),& &x\in\mathbb{R},\\
&u(0,x)=u_0(x),~h(0)=-g(0)=h_0,& &x\in[-h_0,h_0]
\end{aligned}
\right.
\label{100100}
\end{equation}
with
$f_1(0,v)=f_2(u,0)=0$ for any $u,v\in\mathbb{R}$. Following are some imposed assumptions on reaction terms $f_1(u,v)$ and $f_2(u,v)$:

\begin{description}
\item[(A1)]there is constant $K_0>0$ such that $f_1(u,v)<0$ for $u>K_0$ and
$f_2(u,v)<0$ for $v>K_0$;
\item[(A2)]$f_i(u,v)$, $i=1,2$ is locally Lipschitz continuous in $\mathbb{R}^2_+$, i.e., For any $L_i>0, i=1,2$, there exists constant $K_1=K_1(L_i)>0$ such that
$$
|f_i(u_1,v_1)-f_i(u_2,v_2)|\le K_1\left(|u_1-u_2|+|v_1-v_2|\right)~\text{for}~u_i\in[0,L_1],~v_i\in[0,L_2].
$$
\end{description}

The following theorem is the main result of this section.
\begin{theorem}
Assume that {\rm\bf(J)} and {\rm\bf(A1)-(A2)} hold, then for any
given $u_0(x)$ and $v_0(x)$ satisfying (\ref{102}) and
$h_0>0$, problem (\ref{100100}) admits a unique positive
solution $(u,v,g,h)$ defined for all $t>0$. Moreover,
for any given $T>0$,  $u\in\mathbb{X}_{u_0,g,h}$, $v\in\mathbb{X}_{v_0,\infty}$, $g\in\mathbb{G}_{h_0,T}$, $h\in\mathbb{H}_{h_0,T}$.
\label{u-v-g-h-exist}
\end{theorem}

\begin{proof} We divide the proof into two steps.
\begin{description}
\item[Step 1] Local existence and uniqueness.
\end{description}

For any given $v^*(t,x)\in\mathbb{X}_{v_0,\infty}$, it follows from Theorem 2.1 in\cite{CDLL-2018} that for $h_0>0$ and $u_0(x)$ satisfying (\ref{102}) the following problem
\begin{equation}
\left\{
\begin{aligned}
&u_t=d_1\left[\int_{g(t)}^{h(t)}J(x-y)u(t,y)dy-u(t,x)\right]+f_1(u,v^*),
& &t>0,~x\in(g(t),h(t)),\\
&u(t,g(t))=u(t,h(t))=0,& &t>0,\\
&h'(t)=\mu\int_{g(t)}^{h(t)}\int_{h(t)}^{+\infty}J(x-y)u(t,x)dydx,& &t>0,\\
&g'(t)=-\mu\int_{g(t)}^{h(t)}\int_{-\infty}^{g(t)}J(x-y)u(t,x)dydx,& &t>0,\\
&u(0,x)=u_0(x),~h(0)=-g(0)=h_0,& &x\in[-h_0,h_0]
\end{aligned}
\right.
\label{1001}
\end{equation}
admits a unique solution, denoted by $(\tilde u(t,x),\tilde g(t),\tilde h(t))$, which is defined for all $t>0$.
Moreover, for any $T>0$, there hold $\tilde g(t)\in \mathbb G_{h_0, T},\;
 \tilde h(t)\in\mathbb H_{h_0, T}$ and $\tilde u(t,x)\in \mathbb{X}_{u_0, g,h}$.
 Further, we have
 \begin{equation}
0<\tilde u(t,x)\le M:=\max\left\{\max_{-h_0\le x\le h_0}u_0(x),
~K_0\right\}~\mbox{ for } 0<t< T,~x\in(\tilde g(t),\tilde h(t)).
\label{v-bound}
\end{equation}

For the above known $\tilde u(t,x)$,  we first define
\begin{equation}
u^*(t,x)=\left\{
\begin{aligned}
&\tilde u(t,x), & &x\in[\tilde g(t),\tilde h(t)],\\
&0,& &x\not\in[\tilde g(t),\tilde h(t)]
\end{aligned}
\right.
\end{equation}
and then consider the following
\begin{equation}
\left\{
\begin{aligned}
&v_t=d_2\left[\int_{\mathbb{R}}J(x-y)v(t,y)dy-v(t,x)\right]+f_2(u^*,v),
& &0<t<T,~x\in\mathbb{R},\\
&v(0,x)=v_0(x),& &x\in\mathbb{R}.
\end{aligned}
\right.
\label{u-equation}
\end{equation}
The existence and uniqueness of the local solution of (\ref{u-equation}) is well
known, we denote it by $\tilde v(t,x)$, and $\tilde v(t,x)\in C(\Omega_\infty)$.

Next, we want to show that $v^*$ and $\tilde v$ is coincide with each other in the sense that the map $\Gamma$ defined by $\Gamma v^*=\tilde v$ admits a unique fixed point in $\mathbb{X}_{v_0,\infty}$. We prove this conclusion by the contraction mapping theorem. Clearly, $\Gamma$ maps
$\mathbb{X}_{v_0,\infty}$ into itself. It is remained to prove that for small $T>0$, $\Gamma$ is contract.

Firstly we note that $\mathbb{X}_{v_0,\infty}$ is a complete metric space
with the metric
$$
d(\phi_1,\phi_2)=\|\phi_1-\phi_2\|_{L^\infty(\Omega_\infty)}.
$$
Choose $v_i^*\in\mathbb X_{v_0,\infty} (i=1,2)$, we use $\tilde u_i$ to denote the solution of (\ref{1001}) associate to $v_i^*$, and then we use $\tilde v_i$ to denote the solution of (\ref{u-equation}) associate to $u_i^*$. Hence there is
\begin{align*}
\tilde v_1(t,x)-\tilde v_2(t,x)=\int_0^t&d_2\int_{\mathbb{R}}J(x-y)\left[\tilde v_1(s,y)-\tilde v_2(s,y)\right]dyds-\int_0^td_2\left[\tilde v_1-\tilde v_2\right](s,x)ds\\
&+\int_0^t\left[f_2(u_1^*,\tilde v_1)-f_2(u_2^*,\tilde v_2)\right](s,x)ds.
\end{align*}
It then follows that
\begin{align*}
\|\tilde v_1-\tilde v_2\|_{C(\Omega_\infty)}&\le 2 d_2\|\tilde v_1-\tilde v_2\|_{C(\Omega_\infty)}t+K_1\left[\|\tilde v_1-\tilde v_2\|_{C(\Omega_\infty)}+\|u_1^*-u_2^*\|_{C(\Omega_\infty)}\right]t\\
&\le\left(2 d_2+K_1\right)\|\tilde v_1-\tilde v_2\|_{C(\Omega_\infty)}T
+K_1\|\tilde u_1-\tilde u_2\|_{C(\Omega_{\tilde g,\tilde h})}T,
\end{align*}
Taking $T$ sufficiently small such that
$$
\left(2 d_2+K_1\right)T\le\frac 12, ~~\mbox{that is}~T\le T_1:=\frac{1}{2\left(2 d_2+K_1\right)},
$$
then we obtain that
\begin{equation}
\|\tilde v_1-\tilde v_2\|_{C(\Omega_\infty)}\le 2K_1\|\tilde u_1-\tilde u_2\|_{C(\Omega_{\tilde g,\tilde h})}T.
\label{U-estimate-initial}
\end{equation}

We are now in a position to give an estimate to $\|\tilde u_1-\tilde u_2\|_{C(\Omega_{\tilde g,\tilde h})}$. For any given $(t^*,x^*)\in\Omega_{\tilde g,\tilde h}$ and $t\in(0,T]$, define
$$
U(t,x^*)=\tilde u_1(t,x^*)-\tilde u_2(t,x^*),~~~~~V(t,x^*)=\tilde v_1(t,x^*)-\tilde v_2(t,x^*)$$
 and
 \begin{equation}
t_x=\left\{
\begin{aligned}
&t_{x,\tilde g}& &\text{if}~x\in[\tilde g(T),-h_0)~\text{and}~x=\tilde g(t_{x,\tilde g}),\\
&0& &\text{if}~x\in[-h_0,h_0],\\
&t_{x,\tilde h}& &\text{if}~x\in(h_0,\tilde h(T)]~\text{and}~x=\tilde h(t_{x,\tilde h}).
\end{aligned}
\right.
\label{t_x-definition}
\end{equation}

We also define
\begin{align*}
&H_1(t)=\min\left\{\tilde h_1(t),~\tilde h_2(t)\right\},~~~H_2(t)=\max\left\{
\tilde h_1(t),~\tilde h_2(t)\right\},\\
&G_1(t)=\min\left\{\tilde g_1(t),~\tilde g_2(t)\right\},~~~~~~G_2(t)=\max\left\{
\tilde g_1(t),~\tilde g_2(t)\right\},\\
&\Omega_T=\Omega_{G_1, H_2}=\Omega_{\tilde g_1, \tilde h_1}\cup \Omega_{\tilde g_2, \tilde h_2}.
\end{align*}
Three cases will be handled separately..
\begin{description}
\item[Case 1] $x^*\in[-h_0, h_0]$.
\end{description}

By the equations  of $\tilde u_i(t,x^*)$ we have that
\begin{equation}
\left\{
\begin{aligned}
&U_t(t,x^*)+c_1(t,x^*)U(t,x^*)=c_2(t,x^*)V^*(t,x^*)+I(t,x^*),\\
&U(0,x^*)=0,
\end{aligned}
\right.
\label{U-V}
\end{equation}
where $V^*(t,x^*)=v_1^*(t,x^*)-v_2^*(t,x^*)$,
\begin{equation}
c_1(t,x^*)=d_1-\frac{f_1(\tilde u_1,v_1^*)-f_1(\tilde u_2,v_1^*)}{\tilde u_1-\tilde u_2},~~~
c_2(t,x^*)=\frac{f_1(\tilde u_2,v_1^*)-f_1(\tilde u_2,v_2^*)}{v_1^*-v_2^*}
\label{definition-c(t,x)}
\end{equation}
and
\begin{align*}
I(t,x^*)=d_1\int_{\tilde g_1(t)}^{\tilde h_1(t)}J(x^*-y)\tilde u_1(t,y)dy
-d_1\int_{\tilde g_2(t)}^{\tilde h_2(t)}J(x^*-y)\tilde u_2(t,y)dy
\end{align*}
Clearly, at $(t,x^*)$ there hold
$$
\|c_1\|_{\infty},~\|c_2\|_{\infty}\le d_1+K_1(M):=K_1^*.
$$
With
$$
\|H^*\|_{C([0,T])}=\|\tilde h_1-\tilde h_2\|_{C([0,T])}+
\|\tilde g_1-\tilde g_2\|_{C([0,T])},~~M_0=\max\left\{\|u_0\|_\infty,\|v_0\|_\infty,
~K_0\right\},
$$
we also have
$$
\Big|I(t,x^*)\Big|\le d_1\|U\|_{C(\Omega_T)}
+d_1M_0\|J\|_\infty\|H^*\|_{C([0,T])}.
$$
Then we can find constant $C_1=C_1(d_1,u_0,M_0,J)$ such that
\begin{equation}
\begin{aligned}
&\max_{t\in[0,T]}\Big|I(t,x^*)\Big|\le C_1\left[
\|U\|_{C(\Omega_T)}+\|H^*\|_{C([0,T])}\right].
\end{aligned}
\label{I-II-estimate}
\end{equation}
In addition, it follows from (\ref{U-V}) that
\begin{align*}
\left[e^{\int_0^tc_1(\tau,x^*)d\tau}U\right]_t
(t,x^*)=e^{\int_0^tc_1(\tau,x^*)d\tau}\left[c_2(t,x^*)V^*(t,x^*)
+I(t,x^*)\right].
\end{align*}
Then integration from $0$ to $t^*$ immediately leads to
\begin{align*}
U(t^*,x^*)
=e^{-\int_0^{t^*}c_1(\tau,x^*)d\tau}
\int_0^{t^*}e^{\int_0^tc_1(\tau,x^*)d\tau}\left[c_2(t,x^*)V^*(t,x^*)+I(t,x^*)\right]dt.
\end{align*}
Hence
$$
\Big|U(t^*,x^*)\Big|\le e^{2K_1^*T}T\left[C_1(
\|U\|_{C(\Omega_T)}+\|H^*\|_{C([0,T])})+K_1^*\|V^*\|_{C(\Omega_\infty)}\right].
$$
Then for constant $\tilde C_1=\tilde C_1(C_1,K_1^*)>0$ there is
\begin{equation}
\|U\|_{C(\Omega_T)}\le \tilde C_1e^{2K_1^*T}T\left[\|U\|_{C(\Omega_T)}+\|H^*\|_{C([0,T])}+\|V^*\|_{C(\Omega_\infty)}\right].
\label{U-V-estimate-1}
\end{equation}

\begin{description}
\item[Case 2] $x^*\in(h_0, H_1(s))$.
\end{description}

In such a case, there exist $t_1^*,t_2^*\in(0,t^*)$ such
that $x^*=h_1(t_1^*)=h_2(t_2^*)$. Without loss of generality,
suppose that $0<t_1^*\le t_2^*$. According to (\ref{U-V}),
it is routine to check that
$$
U(t^*,x^*)=e^{-\int_{t_2^*}^{t^*}c_1(\tau,x^*)d\tau}
\left[U(t_2^*,x^*)+\int_{t_2^*}^{t^*}e^{\int_{t_2^*}^t
c_1(\tau,x^*)d\tau}\left[c_2(t,x^*)V^*(t,x^*)+I(t,x^*)\right]dt\right].
$$
Then we have
\begin{equation}
\begin{aligned}
\Big|U(t^*,x^*)\Big|&\le e^{K_1^*t^*}\left[\Big|U(t_2^*,x^*)
\Big|+\int_{t_2^*}^{t^*}e^{K_1^*t}\Big|K_1^*V^*(t,x^*)+I(t,x^*)
\Big|dt\right]\\
&\le e^{K_1^*T}\Big|U(t_2^*,x^*)\Big|+Te^{2K_1^*T}
\left[\max_{t\in[0,T]}\Big|I(t,x^*)\Big|+K_1^*\|V^*\|_{C(\Omega_\infty)}\right]\\
&\le e^{K_1^*T}\Big|U(t_2^*,x^*)\Big|+e^{2K_1^*T}
C_1\|H^*\|_{C([0,T])}T\\
&\quad\quad\quad\quad\quad\quad\quad\quad\quad\quad\quad+e^{2K_1^*T}\left[K_1^*
\|V^*\|_{C(\Omega_\infty)}+C_1\|U\|_{C(\Omega_T)}\right]T.
\end{aligned}
\label{U-estimate-2}
\end{equation}
Notice that
\begin{align*}
U(t_2^*,x^*)&=\tilde u_1(t_2^*,x^*)-\tilde u_1(t_1^*,x^*)+\tilde u_1(t_1^*,x^*)
-\tilde u_2(t_2^*,x^*)\\
&=\tilde u_1(t_2^*,x^*)-\tilde u_1(t_1^*,x^*)+\tilde u_1(t_1^*,h_1(t_1^*))-\tilde u_2(t_2^*,h_2(t_2^*))\\
&=\tilde u_1(t_2^*,x^*)-\tilde u_1(t_1^*,x^*)=\int_{t_1^*}^{t_2^*}(\tilde u_1)_t(t,x^*)dt.
\end{align*}
Using $(A2)$ to conclude that there exists a constant
$C_3=C_3(d_1,M_0,f_1)$ such that
\begin{align*}
\Big|U(t_2^*,x^*)\Big|&\le\int_{t_1^*}^{t_2^*}\left|d_1
\int_{\tilde g_1(t)}^{\tilde h_1(t)}J(x^*-y)\tilde u_1(t,y)dy-d_1\tilde u_1(t,x^*)+
f_1(\tilde u_1(t,x^*),\tilde v_1(t,x^*))\right|dt\\
&\le C_3\left(t_2^*-t_1^*\right).
\end{align*}
Obviously, we have $U(t_2^*,x^*)=0$ if $t_1^*=t_2^*$.
If $t_1^*<t_2^*$, it then follows from $\frac{\tilde h_1(t_2^*)
-\tilde h_1(t_1^*)}{t_2^*-t_1^*}\ge\mu c_1^*$(please see the proof of Theorem 2.1 for the detailed formula of $c_1^*$) that
$$
t_2^*-t_1^*\le\Big|\tilde h_1(t_2^*)-\tilde h_1(t_1^*)\Big|\left(
\mu c_1^*\right)^{-1}.
$$
In addition, there hold
$$
0=\tilde h_1(t_1^*)-\tilde h_2(t_2^*)=\tilde h_1(t_1^*)-\tilde h_1(t_2^*)+\tilde h_1(t_2^*)-h_2(t_2^*).
$$
And hence $\tilde h_1(t_2^*)-\tilde h_1(t_1^*)=\tilde h_1(t_2^*)-\tilde h_2(t_2^*)$.
$$
t_2^*-t_1^*\le\Big|\tilde h_1(t_2^*)-\tilde h_1(t_1^*)\Big|\left(\mu c_1^*
\right)^{-1}=\Big|\tilde h_1(t_2^*)-\tilde h_2(t_2^*)\Big|\left(
\mu c_1^*\right)^{-1},
$$
which further indicates that we can find a constant
$C_4=C_4(\mu c_1^*,C_3)$ such that
\begin{equation}
\Big|U(t_2^*,x^*)\Big|\le C_4\|\tilde h_1-\tilde h_2\|_{C([0,T])}.
\label{U-estimate-interior}
\end{equation}
Now, (\ref{U-estimate-2}) coupled with (\ref{U-estimate-interior}) deduces that
\begin{equation*}
\begin{aligned}
\Big|U(t^*,x^*)\Big|
&\le e^{K_1^*T}C_4\|\tilde h_1-\tilde h_2\|_{C([0,T])}+e^{2K_1^*T}
C_1\|H^*\|_{C([0,T])}T\\
&\quad\quad\quad\quad\quad\quad\quad\quad+e^{2K_1^*T}\left[
K_1^*\|V^*\|_{C(\Omega_\infty)}+C_1\|U\|_{C(\Omega_T)}\right]T.
\end{aligned}
\end{equation*}
Again we can choose constant $\tilde C_4=\tilde C_4(
K_1^*,C_1,C_4)>0$ such that
\begin{equation}
\|U\|_{C(\Omega_T)}\le\tilde C_4e^{2K_1^*T}\left[\|\tilde h_1-\tilde h_2\|_{C([0,T])}+\|H^*\|_{C([0,T])}T+\|V^*\|_{C(
\Omega_\infty)}T+\|U\|_{C(\Omega_T)}T\right].
\label{U-estimate-3-2}
\end{equation}

\begin{description}
\item[Case 3] $x^*\in[H_1(s),H_2(s))$.
\end{description}

Without loss of generality, assume that $h_1(s)<h_2(s)$,
then $H_1(s)=h_1(s)$, $H_2(s)=h_2(s)$, $\tilde u_1(t,x^*)
=0$ for $t\in[t_2^*,t^*]$, and
$$
0<h_2(t^*)-h_2(t_2^*)\le h_2(t^*)-h_1(t^*).
$$
Then we have
\begin{align*}
\tilde u_2(t^*,x^*)&=\int_{t_2^*}^{t^*}\left[d_1\int_{g_2(t)}^{h_2(t)}
J(x^*-y)\tilde u_2(t,y)dy-d_1\tilde u_2(t,x^*)+f_1(\tilde u_2(t,x^*),v_2^*(t,x^*))
\right]dt\\
&\le M_0\left[d_1+K(M_0)\right](t^*-t_2^*)\\
&\le M_0\left[d_1+K(M_0)\right]\left[\tilde h_2(t^*)-\tilde h_2(t_2^*)\right]
\left(\mu c_1^*\right)^{-1}\\
&\le M_0\left[d_1+K(M_0)\right]\left[\tilde h_2(t^*)-\tilde h_1(t^*)
\right]\left(\mu c_1^*\right)^{-1}\\
&\le C_5\|\tilde h_1-\tilde h_2\|_{C([0,T])},
\end{align*}
here constant $C_5=M_0\left[d_1+K(M_0)\right]\left(
\mu c_1^*\right)^{-1}$. And hence
\begin{equation}
\Big|U(t^*,x^*)\Big|=\Big|\tilde u_2(t^*,x^*)\Big|\le C_5\|\tilde h_1-\tilde h_2\|_{C([0,T])}.
\label{U-estimate-4}
\end{equation}
Thus, there is
\begin{equation}
\|U\|_{C(\Omega_T)}\le C_5\|\tilde h_1-\tilde h_2\|_{C([0,T])}.
\label{U-V-estimate-3}
\end{equation}
Therefore, by (\ref{U-V-estimate-1}), (\ref{U-estimate-3-2})
and (\ref{U-V-estimate-3}), we can find a constant
$C_6$ that depends on $(u_0,v_0,d_1,\mu c_1^*,f_1,J)$ such that, whether we are in Case 1, 2
or 3, for $T\le1$, we always have
\begin{equation}
\begin{aligned}
\|U\|_{C(\Omega_T)}&\le C_6e^{2K_1^*T}\left[\|H^*\|_{C([0,T])}+\|U\|_{C(\Omega_T)}T+\|V^*\|_{C(\Omega_\infty)}T\right]
\end{aligned}
\label{U-estimate}
\end{equation}

For the case that $x^*\in(G_2(s),-h_0)$ or $x^*\in(
G_1(s),G_2(s)]$, similarly, there exists  a constant
$\tilde C_6=\tilde C_6(u_0,v_0,d_1,\mu c_1^*,
f_1,J)>0$ such that~(\ref{U-estimate}) holds
with $C_6$ replaced by $\tilde C_6$.

Now let $C^*=\max\left\{C_6, \tilde C_6\right\}$, we
then have
\begin{equation}
\|U\|_{C(\Omega_T)}\le C^*e^{2K_1^*T}\left[\|H^*\|_{C([0,T])}+\|U\|_{C(\Omega_T)}T+\|V^*\|_{C(\Omega_\infty)}T\right].
\label{U-estimate-101}
\end{equation}

We are now in a position to give an estimate for $\|H^*\|_{C([0,T])}$, that is
$\|\tilde h_1-\tilde h_2\|_{C([0,T])}+\|\tilde g_1-\tilde g_2\|_{C([0,T])}$. Recalling that for $i=1,2$,
\begin{equation}
\left\{
\begin{aligned}
&\tilde h'_i(t)= \mu \int_{\tilde g_i(t)}^{\tilde h_i(t)}\int_{\tilde h_i(t)
}^{+\infty}J(x-y)dy\tilde u_i(t,x)dx,\\
&\tilde g'_i(t)=-\mu\int_{\tilde g_i(t)}^{\tilde h_i(t)}\int_{-\infty}^{\tilde g_i(t)}
J(x-y)dy\tilde u_i(t,x)dx.
\end{aligned}
\right.
\label{g-h-formula}
\end{equation}
Following the routine of the proof of Theorem 2.1 in\cite{CDLL-2018}, we obtain that
\begin{align*}
~&\left|\tilde h_1(t)-\tilde h_2(t)\right|
\\
\le~&\mu\int_0^t\left|\int_{\tilde g_1(\tau)}^{\tilde h_1(\tau)}\int_{\tilde h_1(\tau)
}^{+\infty}J(x-y)\tilde u_1(\tau,x)dydxd\tau-\int_{\tilde g_2(\tau)}^{\tilde h_2(\tau)}
\int_{\tilde h_2(\tau)}^{+\infty}J(x-y)\tilde u_2(\tau,x)dydx\right|d\tau
\\
\le~&\mu\int_0^t\int_{\tilde g_1(\tau)}^{\tilde h_1(\tau)}\int_{\tilde h_1(\tau)
}^{+\infty}J(x-y)\Big|\tilde u_1(\tau,x)-\tilde u_2(\tau,x)\Big|dydx
\\
&~+\mu\int_0^t\left|\left(\int_{\tilde h_1(\tau)}^{\tilde h_2(\tau)}\int_{\tilde h_1(\tau)}^{+\infty}
+\int_{\tilde g_2(\tau)}^{\tilde g_1(\tau)}\int_{\tilde h_1(\tau)
}^{+\infty}+\int_{\tilde g_2(\tau)}^{\tilde h_2(\tau)}\int_{\tilde h_1(\tau)}^{\tilde h_2(\tau)}\right)
J(x-y)\tilde u_2(t,x)dydx\right|d\tau
\\
\le~&3h_0\mu\|\tilde u_1-\tilde u_2\|_{C(\Omega_T)}T+\mu M_0\big(1+3h_0\|J\|_\infty\big)\|\tilde h_1-\tilde h_2\|_{C([0,T])}T
+\mu M_0\|\tilde g_1-\tilde g_2\|_{C([0,T])}T\\
\le ~ &C_0T\Big[\|\tilde u_1-\tilde u_2\|_{C(\Omega_T)}+\|\tilde h_1-\tilde h_2\|_{C([0,T])}+\|\tilde g_1-\tilde g_2\|_{C([0,T])}\Big],
\end{align*}
where $C_0$ depends only on $(h_0,\mu, u_0, J, M_0)$. Let us recall that $\tilde u_i$ is always extended by 0 in $\big([0,\infty)\times \mathbb R\big)\setminus \Omega_{\tilde g_i, \tilde h_i}$ for $i=1,2$.

Similarly, we have, for $t\in [0, T]$,
\begin{align*}
\Big|\tilde g_1(t)-\tilde g_2(t)\Big|\le C_0T\Big[\|\tilde u_1-\tilde u_2\|_{C(\Omega_T)}+\|\tilde h_1-\tilde h_2\|_{C([0,T])}+\|\tilde g_1-\tilde g_2\|_{C([0,T])}\Big].
\end{align*}
Therefore,
\begin{equation*}
\begin{aligned}
&~\|\tilde h_1-\tilde h_2\|_{C([0,T])}+\|\tilde g_1-\tilde g_2
\|_{C([0,T])}\\
\le&~2C_0T\left[\|\tilde u_1-\tilde u_2\|_{C(\Omega_T)}+\|\tilde h_1-\tilde h_2\|_{C([0,T])}+\|\tilde g_1-\tilde g_2\|_{C([0,T])}\right].
\end{aligned}
\label{20013}
\end{equation*}
If we choose $\delta_3>0$ small such that $2C_0\delta_3\le\frac 12$, then there is
\begin{equation}
\|\tilde h_1-\tilde h_2\|_{C([0,T])}+\|\tilde g_1-\tilde g_2
\|_{C([0,T])}\le4C_0T\|\tilde u_1-\tilde u_2\|_{C(\Omega_T)}~\text{for}~
T\le\delta_3.
\label{200131}
\end{equation}
Substituting (\ref{200131}) into inequality (\ref{U-estimate-101}) leads to the following
\begin{equation*}
\|U\|_{C(\Omega_T)}\le C^*e^{2K_1^*T}\left[(4C_0T+1)\|U\|_{C(\Omega_T)}+\|V^*\|_{C(\Omega_\infty)}\right]T
\end{equation*}
If we choose $\delta_4>0$ small such that
$$
C^*e^{2K_1^*\delta_4}(2C_0\delta_4+1)\delta_4\le\frac 12,
$$
then for $T\le\delta_4$, there is
\begin{equation*}
\|U\|_{C(\Omega_T)}\le \|V^*\|_{C(\Omega_\infty)}.
\end{equation*}

Back to (\ref{U-estimate-initial}), we obtain that
$$
\|\tilde v_1-\tilde v_2\|_{C(\Omega_\infty)}\le 2K_1\|\tilde u_1-\tilde u_2\|_{C(\Omega_T)}T\le2K_1 \|v_1^*-v_2^*\|_{C(\Omega_\infty)}T.
$$
Therefore, for $T<\min\left\{\frac{1}{4K_1}, T_1, \delta_1,\delta_2,\delta_3,\delta_4,1\right\}$, $\Gamma$ is a contraction mapping.
\begin{description}
\item[Step 2] Global existence and uniqueness.
\end{description}

It follows from Step 1 that problem (\ref{100100}) admits
a unique solution $(\tilde u,\tilde v,\tilde g,\tilde h)$ that define for $t\in(0,T)$, and for any given $s\in(0,T)$, there is $u(s,x)>0$ for all $x\in(g(s),h(s))$, and $v(s,x)>0$ for $x\in\mathbb{R}$. Also, $u(s,\cdot)$ and $v(s,\cdot)$ are continuous in $[g(s),h(s)]$ and $\mathbb{R}$ respectively. So if we use
$u(s,\cdot),v(s,\cdot)$ as the initial functions
and then repeat the above Step 1, the solution of (\ref{100100}) can
be extended from $t=s$ to some $T'\ge T$. Through this
extension procedure, we assume that $(0,T_{\max})$ is the
maximum existence interval of which that $(\tilde u,\tilde v,\tilde g,\tilde h)$ can be defined,
below we will prove that $T_{\max}=\infty$.

We will derive this by a contradiction. Suppose on the contrary that
$T_{\max}\in(0,\infty)$, note that
\begin{align*}
\tilde h'(t)-\tilde g'(t)= \mu \int_{\tilde g(t)}^{\tilde h(t)}\left[\int_{\tilde h(t)
}^{+\infty}+\int_{-\infty}^{\tilde g(t)}\right]J(x-y)\tilde u(t,x)dydx\le\mu M_0\left[\tilde h(t)-\tilde g(t)\right],
\end{align*}
which in turn deduces that
\begin{align*}
\tilde h(t)-\tilde g(t)\le2h_0e^{\mu M_0t}\le2h_0e^{\mu M_0T_{\max}}.
\end{align*}
Since $\tilde h(t)$ and $\tilde g(t)$ are monotone
for $t\in[0,T_{\max})$, so we define
$$
\tilde h(T_{\max})=\lim_{t\rightarrow T_{\max}}\tilde h(t),~~\tilde g(T_{\max})=\lim_{t\rightarrow T_{\max}}\tilde g(t), ~\text{ then } \tilde h(T_{\max})-\tilde g(T_{\max})\le2h_0e^{\mu M_0T_{\max}}.
$$

The free boundary conditions (resp. the third
and forth equations) in problem (\ref{101}), together with the
conclusion $0<\tilde u(t,x),\tilde v(t,x)\le M_0$ implies that
$\tilde h'(t),\tilde g'(t)\in L^\infty([0,T_{\max}))$. And hence the
definitions of $\tilde g(T_{\max})$ and $\tilde h(T_{\max})$ show that
$\tilde h(t), \tilde g(t)\in C([0,T_{\max}])$.

Also, we know that
\begin{align*}
&d_1\left[\int_{g(t)}^{h(t)}J(x-y)u(t,y)dy-u(t,x)\right]+f_1(u,v)
\in L^\infty(\Omega_{F}^{T_{\max}}),\\
&d_2\left[\int_{\mathbb{R}}J(x-y)v(t,y)dy-v(t,x)\right]+f_2(u,v)
\in L^\infty(\Omega_{\infty}^{T_{\max}}),
\end{align*}
where $\Omega_{F}^{T_{\max}}=\{(t,x)\in\mathbb{R}^2:t\in[0,T_{\max}],~x\in
(\tilde g(t),\tilde h(t))\}$ and $\Omega_{\infty}^{T_{\max}}=\{(t,x)\in\mathbb{R}^2: t\in[0,T_{\max}],~x\in\mathbb{R}\}$. Then $\tilde u_t(t,x)\in L^\infty(
\Omega_{F}^{T_{\max}})$ and $\tilde v(t,x)\in L^\infty(\Omega_{\infty}^{T_{\max}})$. Hence for each $x\in(\tilde g(T_{\max}),\tilde h(T_{\max}))$,
$\tilde u(T_{\max},x):=\lim_{t\nearrow~T_{\max}}\tilde u(t,x)$ exists, and for each $x\in{\mathbb{R}}$, $\tilde v(T_{\max},x)=
\lim_{t\nearrow~T_{\max}}\tilde v(t,x)$ exists. In addition $
\tilde u(\cdot,x)$ and $\tilde v(\cdot,x)$ are continuous for $t=T_{\max}$.

Now for the known $(\tilde v, \tilde g,\tilde h)$ and $t_x$ defined in (\ref{t_x-definition}) with $T$ replaced by $T_{\max}$, if we regard $\tilde u$ as the unique solution of the following ODEs
\begin{equation*}
\left\{
\begin{aligned}
 &\tilde u_t=d_1\int_{\tilde g(t)}^{\tilde h(t)}J(x-y)\phi(t,y)dy-d_1\tilde u(t,x)+f_1(u,\tilde v),& &t_x<t\le T_{\max},\\
&\tilde u(t_x,x)=\tilde u_0(x),& &x\in(\tilde g(T_{\max}),\tilde h(T_{\max}))
\end{aligned}
\right.
\end{equation*}
with $\phi=\tilde u$, and $\tilde u_0(x)=u_0(x)$ if $x\in[-h_0,h_0]$ and $\tilde u_0(x)=0$ if $x\not\in[-h_0,h_0]$, since $t_x$, $J(\cdot)$
and $f_1(u,\tilde v)$ all are continuous functions
of the spacial variable $x$, then it follows from the
continuous dependence of the ODE solution on the initial
function and the parameters involved in the equation
that the pair $\tilde u$ is continuous in $\Omega_{F}^{T_{\max}}$.
Hence for any $s\in(0,T_{\max})$, we get that $\tilde u\in C(
\bar\Omega_{F}^s)$. For such a $s$, we can get that $\tilde v\in C(\Omega_\infty^s)$ in a paralle way.

Next, we will prove that $(\tilde u,\tilde v)$ is continuous at $t=T_{\max}$.
For the continuity of $\tilde u$, it is suffices to prove that $\tilde u(t,x)\rightarrow0$ for $(t,x)\rightarrow(T_{\max},\tilde g(T_{\max}))$ and $(t,x)\rightarrow(T_{\max}, \tilde h(T_{\max}))$.
We just show the case that $(t,x)\rightarrow(T_{\max},\tilde g(T_{\max}))$, the remained one can be get analogously. As we see below,  if $(t,x)\rightarrow(T_{\max},\tilde g(T_{\max}))$, there hold
\begin{align*}
\left|\tilde u(t,x)\right|&=\left|\int_{t_x}^t\left[d_1
\int_{\tilde g(\tau)}^{\tilde h(\tau)}J(x-y)\tilde u(\tau,y)dy-d_1\tilde u(\tau,x)
+f_1(\tilde u,\tilde v)\right]d\tau\right|\\
&\le\left(t-t_x\right)\left[2d_1+K_1\right]M_0\rightarrow0
\end{align*}
since $t_x\rightarrow T_{\max}$ if $x\rightarrow\tilde g(T_{\max})$.
In addition, as $t_x\rightarrow T_{\max}$, there holds, for each $x\in\mathbb{R}$,
$$
\left|\tilde v(t,x)-\tilde v(t_x,x)\right|\le\left(t-t_x\right)\left[2d_2
+K_1\right]M_0\rightarrow0.
$$
Then $(\tilde u,\tilde v)\in C(\bar\Omega_{F}^{T_{\max}})
\times C(\Omega_{\infty}^{T_{\max}})$, and $(\tilde u,\tilde v,\tilde g,\tilde h)$ verifies problem (\ref{101}) for $t\in(0,T_{\max})$. Again, Lemma \ref{Maximum-Principle} shows that
$\tilde u(T_{\max},x)>0, \tilde v(T_{\max},x)>0$ in $(\tilde g(T_{\max}),
\tilde h(T_{\max}))$. Now if we use $(\tilde u(T_{\max},x), \tilde v(T_{\max},x))$ as the initial function and then take Step 1, so the solution of
(\ref{101}) can be extended to interval $(0,\tilde T)$
with $\tilde T>T_{\max}$, which contradicts to the definition
of $T_{\max}$. Then $T_{\max}=\infty$ follows.
This completes the proof.
\end{proof}

Following is a comparison principle for the competitive model.
\begin{theorem}
{\rm (Comparison principle)}
Assume that {\rm\bf(J)} and {\rm\bf(A1)-(A2)} hold. For $T\in(0,\infty)$, suppose that $\overline h,\overline
g,\underline h,\underline g\in C([0,T])$, $\overline u\in C(D^*_T)\cap C(\overline D^*_T)$ with $D^*_T=\left\{
0<t\leq T,\overline g(t)<x<\overline h(t)\right\}$,
$\underline u\in C(D^{**}_T)\cap C(\overline D^{**}_T)$
with $D^{**}_T=\left\{0<t\leq T,~\underline g(t)<x<\underline h(t)\right\}$,
and $\overline v,\underline v\in (C\cap L^\infty)([0,T]\times\mathbb{R})$ satisfying
\begin{equation}
\left\{
\begin{aligned}
&\overline u_t\ge d_1\left[\int_{\overline g(t)}^{\overline h(t)}J(x-y)\overline u(t,y)dy-\overline u\right]+f_1(\overline u,\underline v),
& &0<t\le T,~x\in(\overline g(t),\overline h(t)),\\
&\overline v_t\ge d_2\left[\int_{\mathbb{R}}J(x-y)\overline v(t,y)dy-\overline v\right]+f_2(\underline u,\overline v),
& &0<t\le T,~x\in\mathbb{R},\\
&\underline u_t\le d_1\left[\int_{\underline g(t)}^{\underline h(t)}J(x-y)\underline u(t,y)dy-\underline u \right]+f_1(\underline u,\overline v),
& &0<t\le T,~x\in(\underline g(t),\underline h(t)),\\
&\underline v_t\le d_2\left[\int_{\mathbb{R}}J(x-y)\underline v(t,y)dy-\underline v\right]+f_2(\overline u,\underline v),
& &0<t\le T,~x\in\mathbb{R}
\end{aligned}
\right.
\label{210}
\end{equation}
with $\underline v(0,x)\le v_0(x)\le\overline v(0,x)$ in $\mathbb{R}$, $\overline u(0,x)\ge u_0(x)$ in $[-h_0,h_0]$, $\underline u(0,x)\le u_0(x)$ in $[\underline g(0),\underline h(0)]$, and
\begin{equation}
\left\{
\begin{aligned}
&\overline h'(t)\ge\mu\int_{\overline g(t)}^{\overline h(t)}
\int_{\overline h(t)}^{+\infty}J(x-y)\overline u(t,x)dydx,& &0<t\le T,\\
&\overline g'(t)\le-\mu\int_{\overline g(t)}^{\overline h(t)}
\int_{-\infty}^{\overline g(t)}J(x-y)\overline u(t,x)dydx,& &0<t\le T,\\
&\underline h'(t)\le\mu\int_{\underline g(t)}^{\underline h(t)}
\int_{\underline h(t)}^{+\infty}J(x-y)\underline u(t,x)dydx,& &0<t\le T,\\
&\underline g'(t)\le-\mu\int_{\underline g(t)}^{\underline h(t)}
\int_{-\infty}^{\underline g(t)}J(x-y)\underline u(t,x)dydx,& &0<t\le T
\end{aligned}
\right.
\label{211}
\end{equation}
with $\overline h(0)\ge h_0\ge\underline h(0)$ and
$\overline g(0)\le-h_0\le\underline g(0)$. Further, we assume that
$\overline u(t,x)=0$ if $x\not\in(\overline g(t),\overline h(t))$ and $\underline u(t,x)=0$ if $x\not\in(\underline g(t),\underline h(t))$.
Then the unique solution $(u,v,g,h)$ of problem (\ref{101}) satisfies
\begin{equation}
\left\{
\begin{aligned}
&u\le\overline u,~v\ge\underline v,~g\ge\overline g~\text{ and }~h
\le\overline h~\text{ for }~0<t\leq T~\text{ and }~x\in\mathbb{R},\\
&u\ge\underline u,~v\le\overline v,~g\le\underline g~\text{ and }~h
\ge\underline h~\text{ for }~0<t\leq T~\text{ and }~x\in\mathbb{R}.
\label{212}
\end{aligned}
\right.
\end{equation}
\label{Comparison-Principle-1}
\end{theorem}

\begin{proof}
The idea of the proof comes from \cite[Lemma 2.6]{DYH2014}.
Since the results involving $(\overline u,\underline v,
\overline g,\overline h)$ and $(\underline u,\overline v,
\underline g,\underline h)$ can be obtained in a similar
manner, so we only show the proof of $u\le\overline u,~v\ge
\underline v,~g\ge\overline g$ and $h\le\overline h$. Assume
that $\overline v$ and $v$ are bounded above by $\tilde M$
in $[0,T]\times\mathbb{R}$, letting $w=\tilde M-v$ and
$\overline w=\tilde M-\underline v$, then $(\overline u,
\overline w,\overline g,\overline h)$ satisfies
\begin{equation}
\left\{
\begin{aligned}
&\overline u_t-d_1\left[\int_{\overline g(t)}^{\overline h(t)}J(x-y)\overline u(t,y)dy-\overline u\right]
\ge f_1(\overline u,\tilde M-\overline w),& &0<t\le T,
~x\in(\overline g(t),\overline h(t)),\\
&\overline w_t-d_2\left[\int_{\mathbb{R}}J(x-y)\overline w(t,y)dy-\overline w\right]
\ge-f_2(\overline u,\tilde M-\overline w),& &0<t\le T,~x\in\mathbb{R},\\
&\overline u(0,x)\ge u_0(x),& &x\in[-h_0,h_0],\\
&\overline w(0,x)=\tilde M-\underline v(0,x)\ge\tilde M-v_0(x),
& &x\in\mathbb{R}.
\end{aligned}
\right.
\label{213}
\end{equation}

We now state that $\overline u\ge0$ over the region $D_T^*$
and $\overline w\ge0$ over the region $[0,T]\times\mathbb{R}$. We only give the proof of
$\overline u\ge0$  in $D_T^*$ since the proof for $\overline w\ge0$ is parallel.

Let $u^1(t,x)=\overline u(t,x)e^{k_1t}$,
in which $k_1>0$ is a constant to be determined later. Then
for all $(t,x)\in D_T^*$, there is
\begin{equation}
u^1_t\ge d_1\int_{\overline g(t)}^{\overline h(t)}J(x-y)u^1(t,y)dy+\left[k_1-d_1+f_{1,u}(\eta,\tilde M-\overline w)\right]u^1(t,x),
\label{214}
\end{equation}
where $\eta(t,x)$ is between $\overline u$ and 0. Now, we
choose $k_1$ is large such that $p(t,x)=k_1-d_1+f_{1,u}(\eta,
\tilde M-\overline w)>0$ for all $(t,x)\in D_T^*$.

Taking $p_0=\sup_{(t,x)\in D_T^*}p(t,x)$ and $T_1=\min
\left\{T,~\frac{1}{d_1+p_0}\right\}$. Suppose on the contrary
that there are $\tilde t\in(0,T_1]$ and $\tilde x\in(\overline
g(\tilde t),\overline h(\tilde t))$ such that $u^1(\tilde t,
\tilde x)<0$. Then
$$
u^1_{\min}=\min_{0<t\le T_1,~x\in(\overline g(t),\overline h(t))}u^1(t,x)<0.
$$
Assume that $u^1_{\min}$ is attained at $(t_1,x_1)$ for $t_1\in(0,T_1]$ and
$x_1\in(g(t_1),h(t_1))$. For $0<t\le t_1$ and $x\in[g(t),h(t)]$, define
\begin{equation*}
u^1(t_x,x)=\left\{
\begin{aligned}
&0,& &x\not\in[-h_0,h_0],\\
&u_0(x),& &x\in[-h_0,h_0],
\end{aligned}
\right.
~~~\text{and }~~~t_x=\left\{
\begin{aligned}
&t_{x,g},& &x\in[g(t),-h_0),\\
&0,& &x\in[-h_0,h_0],\\
&t_{x,h},& &x\in(h_0,h(t)],
\end{aligned}
\right.
\end{equation*}
where $0<t_{x,g}<t_1$ and $0<t_{x,h}<t_1$ have the same meaning as them in (\ref{t_x-definition}). Integrating (\ref{214}) from $t_{x_1}$ to $t_1$ yields that
\begin{align*}
u^1(t_1,x_1)-u^1(t_{x_1},x_1)&\ge d_1\int_{t_{x_1}}^{t_1}\int_{\overline g(t)}^{\overline h(t)}J(x_1-y)u^1(t,y)dydt+\int_{t_{x_1}}^{t_1}p(t,x_1)u^1(t,x_1)dt\\
&\ge d_1\int_{t_{x_1}}^{t_1}\int_{\overline g(t)}^{\overline h(t)}J(x_1-y)u^1_{\min}dydt+\int_{t_{x_1}}^{t_1}p(t,x_1)u^1_{\min}dt\\
&\ge\left(t_1-t_{x_1}\right)(d_1+p_0)u^1_{\min}
\end{align*}
Since $u^1(t_{x_1},x_1)=e^{kt_{x_1}}\bar u(t_{x_1},x_1)\ge0$, then
$$
u^1(t_1,x_1)=u^1_{\min}>t_1(d_1+p_0)u^1_{\min}.
$$
And hence $t_1(d_1+p_0)>1$, that is $t_1>\frac{1}{d_1+p_0}$, which contradicts to our choice of $t_1$. It then follows that $u^1(t,x)\ge0$ for all
$(t,x)\in D_{T_1}^*$, and then $\overline u(t,x)\ge0$ for
$(t,x)\in D_T^*$ by repeating this process and each time, the time interval can be extended by $T_1$ units, and then $\bar u(t,x)\ge0$ for all $(t,x)\in D_T^*$.
Analogously, we have $\bar w(t,x)\ge0$ for all $(t,x)\in [0,T]\times\mathbb{R}$.

Suppose that $\underline h(0)<h_0<\overline h(0)$
and $\underline g(0)>-h_0>\overline g(0)$, and claim that
$h(t)<\overline h(t)$ and $g(t)>\overline g(t)$ (resp.
$h(t)>\underline h(t)$ and $g<\underline g(t)$) for all
$t\in(0,T]$.

Clearly, it is true for small $t>0$. If our claim
does not hold, then we can find a first $t^*\in(0,T]$ such
that $h(t)<\overline h(t)$, $g(t)>\overline g(t)$ for all
$t\in(0,t^*)$, and $h(t^*)=\overline h(t^*)$,
$g(t^*)=\overline g(t^*)$ hold (resp. $h(t)>\underline h(t)$,
$g(t)<\underline g(t)$ for all $t\in(0,t^*)$, and
$h(t^*)=\underline h(t^*)$, $g(t^*)=\underline g(t^*)$ hold).

Letting $U(t,x)=\left(\overline
u-u\right)e^{-k_2t}$ and $W(t,x)=\left(\overline w-w\right)e^{-k_2t}$,
then we get
\begin{equation*}
\left\{
\begin{aligned}
&U_t-d_1\left[\int_{g(t)}^{h(t)}J(x-y)U(t,y)dy-U\right]\ge\left(a-k_2\right)U+bW,& &0<t\le t^*,~x\in(g(t),h(t)),\\
&W_t-d_2\left[\int_{\mathbb{R}}J(x-y)W(t,y)-W\right]\ge cU+\left(d-k_2\right)W,
& &0<t\le t^*,~x\in\mathbb{R},\\
&U(0,x)=\overline u(0,x)-u_0(x)\ge0,& &x\in[-h_0,h_0],\\
&W(0,x)=\overline w(0,x)-w(0,x)=v_0(x)-\underline v(0,x)\ge0,
& &x\in\mathbb{R}
\end{aligned}
\right.
\label{215}
\end{equation*}
with
\begin{align*}
&a=a(t,x)=f_{1,u}\left(\theta_1\overline u+(1-\theta_1)u,
\tilde M-\overline w\right)~\text{ for }~0<\theta_1<1,\\
&b=b(t,x)=-f_{1,v}(u,\theta_2(\tilde M-\overline w)+(1-\theta_2)(\tilde M-w))~\text{ for }~0<\theta_2<1,\\
&c=c(t,x)=-f_{2,u}((\theta_3\overline u+(1-\theta_3)u,\tilde M-\overline w))~\text{ for }~0<\theta_3<1,\\
&d=d(t,x)=f_{2,v}(u,\theta_4(\tilde M-\overline w)+(1-\theta_4)(\tilde M-w))~\text{ for }~0<\theta_4<1.
\end{align*}
And $k_2>0$ is sufficiently large such that
$$
k_2\ge1+|a(t,x)|+b(t,x)+c(t,x)+|d(t,x)|\text{ for }0<t\le t^*~\text{and}~x\in\mathbb{R}.
$$
Note that $a(t,x),b(t,x),c(t,x)$ and $d(t,x)$ all are
bounded and $b(t,x),c(t,x)\ge0$ for $0<t\le t^*$ and
$x\in\mathbb{R}$.

For given $l$ with $l>h(t^*)$ and $-l<g(t^*)$, by setting
$$
\overline U(t,x)=U(t,x)+\frac{\tilde M(x^2+t)}{l^2},~~
\overline W(t,x)=W(t,x)+\frac{\tilde M(x^2+t)}{l^2},
$$
then we obtain that
\begin{equation*}
\left\{
\begin{aligned}
&\overline U_t-d_1\left[\int_{g(t)}^{h(t)}J(x-y)\overline U(t,y)dy-\overline U
\right]\ge\left(a-k_2\right)\overline U+b\overline W,
& &0<t\le t^*,~x\in(g(t),h(t)),\\
&\overline W_t-d_2\left[\int_{\mathbb{R}}J(x-y)\overline W(t,y)dy-\overline W\right]\ge c\overline U+\left(d-k_2\right)\overline W,
& &0<t\le t^*,~x\in\mathbb{R},\\
&\overline U(0,x)=U(0,x)+\frac{\tilde Mx^2}{l^2}>0,& &x\in[-h_0,h_0],\\
&\overline W(0,x)=W(0,x)+\frac{\tilde Mx^2}{l^2}>0,& &-l<x<l.
\end{aligned}
\right.
\label{216}
\end{equation*}
Now we will show that
$$
\min\left\{\min_{(t,x)\in[0,t^*]\times[-l,l]}\overline U(t,x),\ \min_{(t,x)\in[0,t^*]\times
[-l,l]}\overline W(t,x)\right\}:=\tau^*\ge0.
$$
If $\tau^*<0$,
then there exist $0<t_1\le t^*$ and $g(t_1)<x_1<h(t_1)$ such
that $\overline U(t_1,x_1)=\tau^*<0$, or there exist $0<t_2
\le t^*$ and $-l<x_2<l$ such that $\overline W(t_2,x_2)=\tau^*<0$.
Assume that the former case occurs, then $\overline U^*(t,x)$
and $\overline W^*(t,x)$ defined respectively by $\overline
U^*=\overline Ue^{k_3t}$ and $\overline W^*=\overline We^{k_3t}$
($k_3$ is a positive constant will be determined later) satisfy
\begin{equation}
\overline U^*_t\ge d_1\int_{g(t)}^{h(t)}J(x-y)\overline U^*(t,y)dy+q(t,x)\overline U^*(t,x)+b(t,x)\overline W^*(t,x),
\label{217}
\end{equation}
with $q(t,x):=k_3-d_1+a(t,x)-k_2$. Now we choose $k_3>0$ large
such that $q(t,x)>0$ for all $(t,x)\in[0,t^*]\times[-l,l]$.

As the proof of $u^1(t,x)\ge0$, taking
$$
q_0=\sup_{(t,x)\in[0,t^*]\times[-l,l]}q(t,x)~\text{ and }~T_2=\min\left\{t_1,~\frac{1}{d_1+q_0+\max_{[0,t^*]\times[-l,l]}b(t,x)}\right\}.
$$
Since $\overline U(t_1,x_1)=\tau^*<0$,
then  we can find $t_3\in(0,T_2)$ such that
$$
\overline U^*_{\inf}=\inf_{0\le t\le t_3,~x\in[-l,l]}\overline U^*(t,x)<0.
$$
It is noticed we can find sequences $t_n\in(0,t_3]$ and
$x_n\in[-l,l]$ such that $\overline U^*(t_n,x_n)\rightarrow
\overline U^*_{\inf}$ as $n\rightarrow\infty$.
Integrating (\ref{217}) from $0$ to $t_n$ yields that
\begin{align*}
&~\overline U^*(t_n,x_n)-\overline U^*(0,x_n)\\
\ge~&d_1\int_0^{t_n}\int_{\mathbb{R}}J(x_n-y)\overline U^*(t,y)dydt+\int_0^{t_n}\left(q(t,x_n)\overline U^*(t,x_n)
+b(t,x_n)\overline W^*(t,x_n)\right)dt\\
\ge~&d_1t_n\overline U^*_{\inf}+q_0t_n\overline U^*_{\inf}
+b(t,x)t_n\overline U^*_{\inf}
\end{align*}
Since $\overline U^*(0,x_1)=\overline U(0,x_1)>0$, then as $n\rightarrow\infty$, there holds
$$
\overline U^*_{\inf}>d_1t_n\overline U^*_{\inf}+q_0t_n
\overline U^*_{\inf}+b(t,x)t_n\overline U^*_{\inf}\ge t_n\left(d_1+q_0
+\max_{[0,t^*]\times[-l,l]}b(t,x)\right)\overline U^*_{\inf}.
$$
That is
$$
t_n\left(d_1+q_0
+\max_{[0,t^*]\times[-l,l]}b(t,x)\right)\ge1,
$$
and hence $t_3\ge\frac{1}{d_1+q_0+\max_{[0,t^*]\times[-l,l]}b(t,x)}$,
which contradicts to our choice
of $t_3$. Then there must hold $\overline U^*\ge0$ and
$\overline U\ge0$ in $[0,t^*]\times [-l,l]$.

For the second case that there exist $0<t_2\le t^*$
and $-l<x<l$ such that $\overline W(t_2,x_2)=\tau^*<0$,
we also can get the same conclusion. From now on, we
have obtained that $\overline U\ge0$ and $\overline W\ge0$
in $[0,t^*]\times [-l,l]$. It then follows that
$$
U(t,x)\ge-\frac{\tilde M(x^2+t)}{l^2}~\text{ and }~
W(t,x)\ge-\frac{\tilde M(x^2+t)}{l^2}
$$
in $[0,t^*]\times [-l,l]$. By taking $l\rightarrow\infty$
immediately yields that $U(t,x)\ge0$ and $W(t,x)\ge0$ for
all $(t,x)\in[0,t^*]\times\mathbb{R}$, and therefore
$\overline u\ge u$ and $\overline w\ge w$ in $[0,t^*]\times
\mathbb{R}$. By applying the above argument over $[0,T]\times
\mathbb{R}$, we have $\overline u\ge u$ and $\overline w\ge w$
in $[0,T]\times\mathbb{R}$.

For $(t,x)\in\Omega_{t^*}:=\left\{(t,x)\in\mathbb{R}^2:
0<t\le t^*,g(t)<x<h(t)\right\}$, letting $Z=\left(\overline
u-u\right)e^{k_4t}$, then as in the proof of $\overline u
\ge0$ we get $Z(t,x)\ge0$ in $\Omega_{t^*}$. Also, there is
$Z(0,x)=\overline u(0,x)-u(0,x)\ge,\not\equiv0$, then $Z(t,x)>0$
and in turn $\overline u(t,x)>u(t,x)$ in $\Omega_{t^*}$.

On
the other hand, we have
\begin{align*}
&0\ge\overline h'(t^*)-h'(t^*)\ge\mu\int_{\overline g(t^*)}^{\overline h(t^*)}\int_{\overline h(t^*)}^{+\infty}J(x-y)\left(
\overline u-u\right)(t^*,x)dydx>0,\\
&0\le\overline g'(t^*)-g'(t^*)\le-\mu\int_{\overline g(t^*)
}^{\overline h(t^*)}\int_{-\infty}^{\overline g(t^*)}J(x-y)
\left(\overline u-u\right)(t^*,x)dydx<0,
\end{align*}
contradictions happen. Then $h(t)<\overline h(t)$ and
$g(t)>\overline g(t)$ for all $t\in(0,T]$ (resp. $h(t)>\underline h(t)$,
$g(t)<\underline g(t)$ for all $t\in(0,T]$ hold), so the claim
is true.

For the case that $h_0=\overline h(0)$ and $-h_0=\overline
g(0)$. Let $(u_\epsilon,v_\epsilon,g_\epsilon,h_\epsilon)$
with $\epsilon>0$ small be the unique positive solution to
(\ref{101}) with $h_0$ and $-h_0$ are respectively replaced
by $h_0(1-\epsilon)$ and $-h_0(1+\epsilon)$. Using the continuous
dependence of solutions on the parameters, we find that
$(u_\epsilon,v_\epsilon,g_\epsilon,h_\epsilon)\rightarrow
(u,v,g,h)$ as $\epsilon\rightarrow0$, and $(u,v,g,h)$ is
the unique solution of (\ref{101}). Then the results can
be obtained by letting $\epsilon\rightarrow0$ in the inequalities $u_\epsilon\le\overline u$, $v_\epsilon\ge\underline v$,
$g_\epsilon>\overline g$ and $h_\epsilon<\overline h$.
\end{proof}

Following are two essential conclusions to be used later.
{\begin{lemma}
\label{lemma-MP2}
{\rm (\cite[Lemma 3.3]{CDLL-2018})}
Assume that {\rm\bf(J)} holds, and  $h_0, T>0$.
Suppose that
 $u(t,x)$ as well as $u_t(t,x)$ are
continuous in $\Omega_0:=[0, T]\times [-h_0, h_0]$, and  for some $c\in L^\infty (\Omega_0)$,
\begin{equation*}\label{lemma-MP2-u}
\left\{
\begin{aligned}
&u_t(t,x)\ge d\int_{-h_0}^{h_0}J(x-y)u(t,y)dy-du +c(t,x)u, && t\in (0, T],\  x\in [-h_0, h_0],\\
&u(0,x)\ge0,  && x\in [-h_0,  h_0].
\end{aligned}
\right.
\end{equation*}
Then $u(t,x)\ge0$ for all $0\le t\le T$ and $x\in[-h_0, h_0]$.
Moreover, if $u(0,x)\not\equiv0$ in $[-h_0, h_0]$, then $u(t,x)>0$ in $(0, T]\times [-h_0, h_0]$.
\end{lemma}}

\section{Dynamics of the two species}
\noindent

In this section, we will devote to the long-term dynamics
of the two species in (\ref{101}) in the case that the
population dynamics of the two species are not identical.
\subsection{Spectrum and principal eigenvalue}
\noindent

In this subsection we collect some essential results
that regarding the following linear dispersal
equation
\begin{equation}
u_t(t,x)=d\left[\int_{\mathbb{R}}J(x-y)u(t,y)dy-u(t,x)+au(t,x)\right]
~\text{for}~t>0~\text{and}~x\in\Omega,
\label{311}
\end{equation}
where $\Omega\in\mathbb{R}$ is an open set, parameter $a>0$ is a constant.
Let $X=C(\overline\Omega,\mathbb{R})$ be equipped with
the maximum norm, $X^+=\left\{u\in X~\Big|~u(x)\ge0,x
\in\overline\Omega\right\}$ and $X^{++}=
\left\{u\in X^+\Big|~u(x)>0,x\in\overline\Omega\right\}$.
For any given $u_1,u_2\in X$, define
\begin{align*}
&u_1\le u_2~\text{ or }~u_2\ge u_1~\text{ if }~u_2-u_1\in X^+,\\
&u_1\ll u_2~\text{ or }~u_2\gg u_1~\text{ if }~u_2-u_1\in X^{++}.
\end{align*}

Following are some well
known results.

\begin{theorem}
\rm{(Coville et al.\cite{CovilleSIAM2008})} Assume that
{\rm\bf(J)} holds. Let $\Omega\subset\mathbb{R}$ is a bounded
open interval. Then there exists a smallest $\lambda_1=
\lambda_1(d,a,\Omega)$ such that problem
\begin{equation}
\left\{
\begin{aligned}
&d\left(\int_{\mathbb{R}}J(x-y)\phi(y)dy-\phi(x)\right)+a\phi(x)=-\lambda_1\phi(x)~\text{ in }\Omega,\\
&\phi=0~\text{ for all }x\not\in\Omega~\text{ and }
~\phi|_{\overline\Omega}~\text{ is continuous}
\end{aligned}
\right.
\label{3113}
\end{equation}
has a nontrivial solution. This eigenvalue is simple and the
eigenfunctions are of constant sign in $\Omega$. Moreover,
$$
\lambda_1(d,a,\Omega)=\min_{\phi\in C(\overline\Omega),\phi\neq0}-
\frac{d\int_{\mathbb{R}}\int_{\mathbb{R}}J(x-y)\tilde\phi(y)\tilde\phi(x)
dydx+a\int_{\mathbb{R}}\tilde\phi^2(x)dx}{\int_{\Omega}\phi^2(x)dx},
$$
where $\tilde\phi$ denotes the extension by 0 of $\phi$
to $\mathbb{R}$ and the minimum is attained.
\label{Thm311}
\end{theorem}

The asymptotic
behavior both for small and lager domains read as
\begin{theorem}
\rm{(\cite[Proposition 3.4]{CDLL-2018})} Assume that {\rm\bf(J)} holds and $\Omega=[l_1,l_2]$ with $-\infty<l_1<l_2<\infty$. Then the principle eigenvalue $\lambda_1(d,a,\Omega)$ of (\ref{3113}) is strictly deceasing and continuous in $l=l_2-l_1$, and\\
(1)~$\lim_{l_2-l_1\rightarrow0}\lambda_1(d,a,\Omega)=d-a$;~~~(2)~$\lim_{l_2-l_1\rightarrow\infty}\lambda_1(d,a,\Omega)=-a$;
\label{Thm315}
\end{theorem}

Theorem \ref{Thm315} immediately deduce the following result.
\begin{theorem}
For the principle eigenvalue $\lambda_1(d,a,\Omega)$ of
(\ref{3113}) with $\Omega=[l_1,l_2]$, if $0<a<d$, then there exists a unique positive $R^*$ such that $\lambda_1(d,a,\Omega)=0$ if $l_2-l_1=R^*$,
$\lambda_1(d,a,\Omega)>0$ if $l_2-l_1<R^*$ and $\lambda_1(d,a,\Omega)<0$ if $l_2-l_1>R^*$.
\label{Thm316}
\end{theorem}

The following results concerns with the asymptotic behavior
of the solution of the evolution problem
\begin{equation}
u_t(t,x)=d(J*u-u)+f(t,u)~\text{ in }~\mathbb{R}^+\times\Omega
~\text{ and }~u(0,x)=u_0(x)~\text{ in }~\Omega,
\label{314}
\end{equation}
where the reaction term $f(x,u)$ satisfying the following assumption:
\begin{description}
\item[(A3)]$f(x,u)\in C(\mathbb{R}\times[0,+\infty))$ is differential with respect to $u$ and $f_u(x,0)$ is Lipshitz continuous in $\mathbb{R}$; $f(x,0)=0$ and $\frac{f(u)}{u}$ is strictly decreasing in $u\in\mathbb{R}^+$; there exists a constant $\tilde K>0$ such that $f(x,u)<0$ for all $x\in\mathbb{R}$ and $u\ge\tilde K$.
\end{description}

\begin{theorem}
\rm{(\cite{BatesJMAA2007,CovilleJDE2010})}  Assume that {\rm\bf(J)} and
{\rm\bf(A3)} hold, and $\Omega\subset\mathbb{R}$ is bounded. Let $u_0$
be an arbitrary bounded and continuous function in $\Omega$
such that $u_0\ge,\not\equiv0$. Let $u(t,x)$ be the solution
of (\ref{314}) with initial datum $u(0,x)=u_0(x)$. Then problem (\ref{314}) admits a unique positive steady state $u_\Omega$ if and only if $\lambda_1(d,f_u(x,0),\Omega)<0$. Further,
\begin{description}
\item[(i)]$u(t,x)\rightarrow u_\Omega$
uniformly in $x\in\Omega$ as $t\rightarrow+\infty$ if $\lambda_1(d,a,\Omega)<0$;
\item[(ii)]$u(t,x)\rightarrow0$ uniformly in $x\in\Omega$ as $t\rightarrow+\infty$ if $\lambda_1(d,a,\Omega)\ge0$.
\end{description}
\label{Thm313}
\end{theorem}

\subsection{Vanishing Case $(h_\infty-g_\infty<\infty)$}
\noindent

It follows from Theorem \ref{u-v-g-h-exist} that $h(t)$ and $-g(t)$ are monotone increasing. Then there exist $h_\infty$ and $g_\infty$ such that
$h_\infty=\lim_{t\rightarrow\infty}h(t)$ and
$g_\infty=\lim_{t\rightarrow\infty}g(t)$. To establish the long time behavior of
$(u,v)$, we first derive an estimate.
\begin{theorem}
Let $(u,v,g,h)$ be the unique solution of (\ref{101}). If $h_\infty-g_\infty<\infty$,
then $\lim_{t\rightarrow\infty}g'(t)=\lim_{t\rightarrow\infty}h'(t)=0$.
\label{h-g-bounded}
\end{theorem}
\begin{proof}
Choose constant $K^*$ with $K^*\ge\max\{\|u\|_{C([g(t),h(t)])},\|u_t\|_{C([g(t),h(t)])}\}$, then for any $\tau,s\ge0$ and $\theta$ between $\tau$ and $s$, we have
\begin{align*}
h'(\tau)-h'(s)&=\mu\int_{g(\tau)}^{h(\tau)}\int_{h(\tau)}^{+\infty}J(x-y)u(\tau,x)
dydx-\mu\int_{g(s)}^{h(s)}\int_{h(s)}^{+\infty}J(x-y)u(s,x)dydx\\
&=\mu\left(\int_{g(\tau)}^{g(s)}+\int_{g(s)}^{h(s)}+\int_{h(s)}^{h(\tau)}\right)\int_{h(\tau)}^{+\infty}J(x-y)u(\tau,x)dydx\\
&\quad\quad\quad\quad\quad\quad\quad\quad\quad\quad
-\mu\int_{g(s)}^{h(s)}\left(\int_{h(s)}^{h(\tau)}+\int_{h(\tau)}^{+\infty}\right)J(x-y)u(s,x)dydx\\
&=\mu\int_{g(s)}^{h(s)}\int_{h(\tau)}^{+\infty}J(x-y)u_t(\theta,x)(\tau-s)dydx
-\mu\int_{g(s)}^{h(s)}\int_{h(s)}^{h(\tau)}J(x-y)u(s,x)dydx\\
&\quad\quad\quad+\mu\int_{g(\tau)}^{g(s)}\int_{h(\tau)}^{+\infty}J(x-y)u(\tau,x)dydx+\mu\int_{h(s)}^{h(\tau)}\int_{h(\tau)}^{+\infty}J(x-y)u(\tau,x)dydx.
\end{align*}
And then for $\xi_1,\xi_2$ between $s$ and $\tau$,
\begin{align*}
|h'(\tau)-h'(s)|&\le\mu K^*|\tau-s|(h(s)-g(s))+\mu K^*|g(\tau)-g(s)|+2\mu K^*|h(\tau)-h(s)|\\
&\le\mu K^*\left[\|g'(\xi_1)\|_\infty+\|h'(\xi_2)\|_\infty+(h_\infty-g_\infty)\right]
|\tau-s|
\end{align*}
along with the condition that $h_\infty-g_\infty<\infty$ indicates that
$h'(t)$ is Lipschitz continuous in $[0,\infty)$. We use the condition $h_\infty-g_\infty<\infty$ again to obtain that $\lim_{t\rightarrow+\infty}h'(t)=0$. Analogously, there is $\lim_{t\rightarrow+\infty}g'(t)=0$.
\end{proof}

\begin{theorem}
Let $(u,v,g,h)$ be the solution of problem (\ref{101}) with $h_\infty-g_\infty
<\infty$, then $\lim_{t\rightarrow+\infty}\|u\|_{C([g(t),h(t)])}=0$.
\label{uv-vanishing}
\end{theorem}
\begin{proof}
Since $h_\infty-g_\infty<\infty$, it follows from Theorem \ref{h-g-bounded} that $h'(t),-g'(t)\rightarrow0$ as $t\rightarrow+\infty$. Also, by $\|u,v\|_{\infty}\le M_0$, there holds
\begin{equation*}
\left\{
\begin{aligned}
&u_t\ge d_1\left[\int_{g(t)}^{h(t)}J(x-y)u(t,y)dy-u\right]+u(a_1-b_1u-c_1M_0),~t>0,~g(t)<x<h(t),\\
&v_t\ge d_2\left[\int_{\mathbb{R}}J(x-y)v(t,y)dy-v\right]+v(a_2-b_2u-c_2M_0),
~t>0,~x\in\mathbb{R}.
\end{aligned}
\right.
\end{equation*}

Assume on the contrary that $\lim_{t\rightarrow+\infty}\|u\|_{C([g(t),h(t)])}>0$, then there exist $\epsilon_1>0$
and sequence
$\left\{(t_k, x_k)\right\}_{k=1}^{\infty}$ with $x_k\in(g(t),h(t))$ and $t_k\rightarrow+\infty$ as $k\rightarrow\infty$ such that $u(t_k,x_k)\ge\frac{\epsilon_1}{2}$
for all $k\in\mathbb{N}$.

Since $g_\infty<g(t)<x_k<h(t)<h_\infty$, passing to a subsequence if necessary, we then have $x_k\rightarrow x_0\in(g_\infty,h_\infty)$
as $k\rightarrow\infty$. For $t\in(-t_k,+\infty)$ and $x\in(g(t+t_k),h(t+t_k))$, define
\begin{align*}
U_k(t,x)=u(t+t_k,x).
\end{align*}
By Theorem \ref{u-v-g-h-exist}, we see that $u(t,x)$ is bounded, it then follows that (passing to a subsequence if necessary)
$U_k(t,x)\rightarrow\tilde U(t,x)$
as $k\rightarrow\infty$, for $x\in(g_\infty,h_\infty)$, $\tilde U(t,x)$ satisfies
\begin{equation*}
\left\{
\begin{aligned}
&\tilde U_t\ge d_1\int_{g_\infty}^{h_\infty}J(x-y)\tilde U(t,y)dy-d_1\tilde U(t,x)+\tilde U(a_1-b_1u-c_1M_0),& &t\in\mathbb{R},\\
&\tilde U(0,x_0)=\lim_{k\rightarrow\infty}U_k(0,x_k)=\lim_{k
\rightarrow\infty}u(t_k,x_k)\ge\frac{\epsilon_1}{2}>0.
\end{aligned}
\right.
\end{equation*}
The Maximum Principle yields that $\tilde U(t,x)>0$ in
$\mathbb{R}\times(g_\infty,h_\infty)$.

Further, since $h'(t),-g'(t)\rightarrow0$ as $t\rightarrow\infty$, then there hold
\begin{align*}
0=\lim_{k\rightarrow\infty}h'(t+t_k)&=\mu\lim_{k\rightarrow\infty}
\int_{g(t+t_k)}^{h(t+t_k)}\int_{h(t+t_k)}^{+\infty}J(x-y)U_k(t,x)dydx\\
&=\mu\int_{g_\infty}^{h_\infty}\int_{h_\infty}^{+\infty}J(x-y)\tilde U(t,x)dydx>0
\end{align*}
and
\begin{align*}
0=\lim_{k\rightarrow\infty}g'(t+t_k)&=-\mu\lim_{k\rightarrow\infty}
\int_{g(t+t_k)}^{h(t+t_k)}\int_{-\infty}^{g(t+t_k)}J(x-y)U_k(t,x)dydx\\
&=-\mu\int_{g_\infty}^{h_\infty}\int_{-\infty}^{g_\infty}J(x-y)\tilde U(t,x)dydx<0,
\end{align*}
contradictions. Hence there holds $\lim_{t\rightarrow+\infty}\|u\|_{C([g(t),h(t)])}=0$. This completes the proof.
\end{proof}

\subsection{Spreading Case $(h_\infty-g_\infty=\infty)$ with $\frac{a_1}{a_2}<\min\left\{\frac{b_1}{b_2},~\frac{c_1}{c_2}\right\}$}
\noindent

The condition $\frac{a_1}{a_2}<\min\left\{\frac{b_1}{b_2},~\frac{c_1}{c_2}\right\}$ here means that when compared with the species $v$, the species $u$ is an inferior competitor. Further, we assume that
\begin{description}
\item[(F1)] $a_2<d_2$.
\end{description}

\begin{theorem}
Assume that $(u,v,g,h)$ is the unique positive solution of (\ref{101}) with
$\frac{a_1}{a_2}<\min\left\{\frac{b_1}{b_2},~\frac{c_1}{c_2}\right\}$. If
$h_\infty-g_\infty=\infty$, then
$\lim_{t\rightarrow+\infty}\left(u(t,x),v(t,x)\right)=\left(0,\frac{a_2}{c_2}\right)$ holds uniformly in any compact subset of $\mathbb{R}$.
\label{Thm325}
\end{theorem}

\begin{proof}
Note that for $t>0$ and $x\in\mathbb{R}$, we have $u(t,x)\le\bar u(t)$ and
$v(t,x)\le\bar v(t)$, here $\bar u(t)$ and $\bar v(t)$ are separately defined by
\begin{equation*}
\left\{
\begin{aligned}
&\bar u(t)=\frac{a_1}{b_1}e^{\frac{a_1}{b_1}t}\left(e^{\frac{a_1}{b_1}t}-1
+\frac{a_1}{b_1\|u_0\|_{L^\infty}}\right)^{-1},\\
&\bar v(t)=\frac{a_2}{c_2}e^{\frac{a_2}{c_2}t}\left(
e^{\frac{a_2}{c_2}t}-1+\frac{a_2}{c_2\|v_0\|_{L^\infty}}\right)^{-1}.
\end{aligned}
\right.
\end{equation*}
Then it is obvious that
$$
\limsup_{t\rightarrow+\infty}u(t,x)\le\frac{a_1}{b_1},~~~
\limsup_{t\rightarrow+\infty}v(t,x)\le\frac{a_2}{c_2}~\text{ uniformly for }x\in\mathbb{R}.
$$
And hence for $0<\epsilon_1
<\frac 12\left(\frac{a_2}{b_2}-\frac{a_1}{b_1}\right)$, we
can find some $t_0>0$ such that $u(t,x)\le\frac{a_1}{b_1}
+\epsilon_1$ for $t\ge t_0$ and $x\in\mathbb{R}$. Then by defining $A_1=b_2\left(\frac{a_2}{b_2}-\frac{a_1}{b_1}
\epsilon_1\right)$, we have
\begin{equation}
\left\{
\begin{aligned}
&v_t-d_2\left[\int_\mathbb{R}J(x-y)v(t,y)dy-v\right]\ge v\left[A_1-c_2v\right],
& &t\ge t_0,~x\in\mathbb{R},\\
&v(t_0,x)>0,& &x\in\mathbb{R}.
\end{aligned}
\right.
\label{319}
\end{equation}
It follows from the comparison principal that $v(t,x)\ge v^*(t,x)$ for all
$t\ge t_0$ and $x\in\mathbb{R}$, where $v^*(t,x)$ is the solution to
\begin{equation}
\left\{
\begin{aligned}
&v^*_t-d_2\left[\int_{\mathbb{R}}J(x-y)v^*(t,y)dy-v^*\right]=v^*\left[A_1-c_2v^*\right],& &t\ge t_0,~x\in\mathbb{R},\\
&v^*(t_0,x)=v(t_0,x)>0,& &x\in\mathbb{R}.
\end{aligned}
\right.
\label{320}
\end{equation}
For given $L>\frac{R^*}{2}$, we can find some $t_L>t_0$ such that $h(t_L)-g(t_L)\ge2L$, and $v^*(t,x)\ge v_L(t,x)$ for $t\ge t_L$ and $x\in(-L,L)$, where $v_L(t,x)$ verifies
\begin{equation}
\left\{
\begin{aligned}
&v_t-d_2\left[\int_{-L}^{L}J(x-y)v(t,y)-v\right]=v\left[A_1-c_2v\right],& &t\ge t_L,~x\in(-L,L),\\
&v(t_L,x)=v^*(t_L,x)>0,& &x\in(-L,L).
\end{aligned}
\right.
\label{v_L-lower-bounded}
\end{equation}
Denoting $\lambda_1(d_2,A_1,\Omega_L)$ with $\Omega_L=[-L,L]$ by the principal eigenvalue of problem (\ref{v_L-lower-bounded}), it then follows from assumption {\rm\bf(F1)} and Theorem \ref{Thm315} that
$$
\lambda_1(d_2,A_1,\Omega_L)<\lambda_1(d_2,A_1,\Omega_{R^*})=0~\text{
with }~|\Omega_{R^*}|=R^*.
$$
And hence it follows from Theorem \ref{Thm313} that
$$
\lim_{t\rightarrow+\infty}v_L(t,x)=\frac{A_1}{c_2}=\frac{b_2}{c_2}\left(\frac{a_2}{b_2}-\frac{a_1}{b_1}-\epsilon_1\right)~\text{ uniformly in any bounded subset of }~\mathbb{R}.
$$
So for the given $L>0$, we can find some $t_L>t_1$ such that
\begin{equation}
v(t,x)\ge v^*(t,x)\ge\tilde A:=\frac{A_1}{2c_2}~\text{ for }~t\ge t_L
~\text{ and }-L\le x\le L.
\label{321}
\end{equation}

Check the equation of $u$, note that $u$ now satisfies
\begin{equation}
\left\{
\begin{aligned}
&u_t-d_1\left[\int_{g(t)}^{h(t)}(x-y)u(t,y)dy-u\right]=u(a_1-b_1u-c_1v),
& &t>t_L,~x\in(g(t),h(t)),\\
&u(t,x)=0,& &t>t_L,~x\not\in(g(t),h(t)),\\
&u(t,x)\le\frac{a_1}{b_1}+\epsilon_1,& &t>t_L,~x\in(g(t),h(t)).
\end{aligned}
\right.
\label{322}
\end{equation}
Then it follows from Comparison
Principle that $u\le\overline u$ and $v\ge\underline v$ for
$t\ge t_L$ and $x\in[-L,L]$, where $(\overline u,\underline v)$
is the solution of
\begin{equation}
\left\{
\begin{aligned}
&\overline u_t-d_1\left[\int_{-L}^{-L}J(x-y)\overline u(t,y)dy-\overline u\right]=\overline u(a_1
-b_1\overline u-c_1\underline v),& &t>t_L,-L<x<L,\\
&\underline v_t-d_2\left[\int_{-L}^{-L}J(x-y)\underline v(t,y)dy-\underline v\right]=\underline v
(a_2-b_2\overline u-c_2\underline v ),& &t>t_L,-L<x<L,\\
&\overline u(t_L,x)=\overline u_{t_L}(x)=\frac{a_1}{b_1}+\epsilon_1,
~\underline v(t_L,x)=\underline v_{t_L}(x)=\tilde A,& &-L\le x\le L,\\
&\overline u(t,-L)=\overline u(t,L)=\frac{a_1}{b_1}+\epsilon_1,
~\underline v(t,L)=\underline v(t,-L)=\tilde A,& &t\ge t_L
\end{aligned}
\right.
\label{323}
\end{equation}
with $\left(\frac{a_1}{b_1}+\epsilon_1,\tilde A\right)$
a pair of upper solution. In view of the dependence of
solutions on initial data, we denote $(u(t,x;u_0,v_0),v
(t,x;u_0,v_0))$ (resp.$(\overline u(t,x;\overline u_{t_L},
\underline v_{t_L}),\underline v(t,x;\overline u_{t_L},
\underline v_{t_L}))$) by the solution of problem (\ref{101})
(resp. (\ref{323})). Note that $f_1^*:=\overline u(a_1-b_1
\overline u-c_1\underline v)$ is nonincreasing in $\underline v$
and $f_2^*:=\underline v(a_2-b_2\overline u-c_2\underline v )$
is nonincreasing in $\overline u$, then (\ref{323}) generates
a monotone dynamical system with respect to the order
$$
(u_1,v_1)\le_2(u_2,v_2)~\text{ if }~u_1\le u_2~\text{ and }~v_1\ge v_2.
$$
This implies that for $t_2>t_1\ge t_L$ and $x\in[-L,L]$,
$$
\left(\overline u(t_2,x;\overline u_{t_L},\underline v_{t_L}),
\underline v(t_2,x;\overline u_{t_L},\underline v_{t_L})\right)
\le_2\left(\overline u(t_1,x;\overline u_{t_L},\underline v_{t_L}),
\underline v(t_1,x;\overline u_{t_L},\underline v_{t_L})\right)
\le_2\left(\frac{a_1}{b_1}+\epsilon_1,\tilde A\right).
$$
Hence
$\lim_{t\rightarrow+\infty}\left(\overline u(t,x;\overline
u_{t_L},\underline v_{t_L}),\underline v(t,x;\overline u_{t_L},
\underline v_{t_L})\right)=\left(\overline u_L(x),\underline v_L
(x)\right)$ uniformly in $[-L,L]$, where $\left(\overline u_L,
\underline v_L\right)$ satisfies
\begin{equation}
\left\{
\begin{aligned}
&-d_1\left[\int_{-L}^LJ(x-y)\overline u_L(y)dy-\overline u_L\right]=
\overline u_L(a_1
-b_1\overline u_L-c_1\underline v_L),& &-L<x<L,\\
&-d_2\left[\int_{-L}^LJ(x-y)\underline v_L(y)dy-\underline v_L\right]=
\underline v_L
(a_2-b_2\overline u_L-c_2\underline v_L ),& &-L<x<L,\\
&\overline u(-L)=\overline u(L)=\frac{a_1}{b_1}+\epsilon_1,
~\underline v(L)=\underline v(-L)=\tilde A.
\end{aligned}
\right.
\label{324}
\end{equation}
By comparing the boundary conditions in $(\ref{324})$
we then observe that for $0<L_1<L_2$, $\overline u_{L_1}
(x)\ge\overline u_{L_2}(x)$ and $\underline v_{L_1}(x)\le
\underline v_{L_2}(x)$ in $[-L_1,L_1]$. Hence, by letting
$L\rightarrow\infty$ and then a diagonal procedure, there
is $\left(\overline u_L(x),\underline v_L(x)\right)\rightarrow
(\overline u^*(x),\underline v^*(x))$ uniformly on any
compact subset of $\mathbb{R}$, where $(\overline u^*,
\underline v^*)$ satisfies
\begin{equation}
\left\{
\begin{aligned}
&-d_1\left[\int_{-L}^LJ(x-y)\overline u^*(y)dy-\overline u^*\right]=
\overline u^*(a_1
-b_1\overline u^*-c_1\underline v^*),& &x\in\mathbb{R},\\
&-d_2\left[\int_{-L}^LJ(x-y)\underline v^*(y)dy-\underline v^*\right]=
\underline v^*
(a_2-b_2\overline u^*-c_2\underline v^* ),& &x\in\mathbb{R},\\
&\overline u^*(x)\le\frac{a_1}{b_1}+\epsilon_1,
~\underline v^*(x)\ge\tilde A,& &x\in\mathbb{R}.
\end{aligned}
\right.
\label{325}
\end{equation}

On the other hand, since $\frac{a_1}{a_2}<\min\left\{
\frac{b_1}{b_2},\frac{c_1}{c_2}\right\}$, then the
solution $(u_1,v_1)$ of the following problem
\begin{equation}
\left\{
\begin{aligned}
&(u_1)_t=u_1(a_1-b_1u_1-c_1v_1),& &t>0,\\
&(v_1)_t=v_1(a_2-b_2u_1-c_2v_1),& &t>0,\\
&u_1(0)=\frac{a_1}{b_1}+\epsilon_1,~v_1(0)=\tilde A,
\end{aligned}
\right.
\label{326}
\end{equation}
satisfies $\lim_{t\rightarrow+\infty}(u_1(t),v_1(t))=\left(
0,\frac{a_2}{c_2}\right)$, see Morita et al.\cite{MoritaSIAM2009}.
This further implies that solution $(U(t,x),V(t,x))$ of the problem
\begin{equation}
\left\{
\begin{aligned}
&U_t-d_1(J*U-U)=U(a_1-b_1U-c_1V),& &t>0,~x\in\mathbb{R},\\
&V_t-d_2(J*V-V)=V(a_2-b_2U-c_2V),& &t>0,~x\in\mathbb{R},\\
&U(0,x)=\frac{a_1}{b_1}+\epsilon_1,~V(0,x)=\tilde A,& &x\in\mathbb{R}
\end{aligned}
\right.
\label{327}
\end{equation}
satisfies $\lim_{t\rightarrow+\infty}\left(U(t,x),V(t,x)
\right)=\left(0,\frac{a_2}{c_2}\right)$ uniformly in any
bounded subset of $\mathbb{R}$. Meanwhile, by using the
comparison principle to problems (\ref{325}) and (\ref{327})
we obtain that
$$
\overline u^*(x)\le U(t,x)~\text{ and }~\underline v^*(x)\ge V(t,x)~\text{ for all }~x\in\mathbb{R},
$$
which indicates
that $\overline u^*(x)=0$ and $\underline v^*(x)\ge\frac{a_2}
{c_2}$ for all $x\in\mathbb{R}$, and then $\overline u_L(x)=0$
and $\underline v_L(x)\ge\frac{a_2}{c_2}$ for $x\in(-L,L)$,
and hence
$$
\overline u(t,x;\overline u_{t_L},\underline v_{t_L})
\rightarrow0~\text{ and }~\underline v(t,x;\overline u_{t_L},\underline
v_{t_L})\ge\frac{a_2}{c_2}~\text{ as }~t\rightarrow+\infty.
$$
Further, we get $u(t,x)=0$ and $v(t,x)\ge\frac{a_2}{c_2}$ as $t\rightarrow
+\infty$. Then
$$
\lim_{t\rightarrow+\infty}\|u(t,\cdot)\|_{C([g(t),h(t)])}=0~\text{ and }~\lim_{
t\rightarrow+\infty}\|v(t,\cdot)\|_{C(\mathbb{R})}=\frac{a_2}{c_2}.
$$
This completes the proof.
\end{proof}

\subsection{Spreading Case $(h_\infty-g_\infty=\infty)$ with $\frac{a_1}{a_2}>\max\left\{\frac{b_1}{b_2},~\frac{c_1}{c_2}\right\}$}
\noindent

The condition $\frac{a_1}{a_2}>\max\left\{\frac{b_1}{b_2},~\frac{c_1}{c_2}\right\}$ here means that when compared with the species $v$, the species $u$ is an superior competitor.

It is stated in section 3.1 that the eigenvalue problem
$$
d_1\left(J-I\right)\phi(x)+a_1\phi(x)=-\lambda_1\phi(x)\text{
in }\Omega,~\phi=0\text{ for all }x\not\in\Omega~\text{ and }
~\phi|_{\overline\Omega}~\text{ is continuous}
$$
admits an eigen pair $(\lambda_1(d_1,a_1,\Omega),\phi_1(x))$.
And if we assume further that $a_1<d_1$, then there exists a
unique $R^*$ such that $\lambda_1(d_1,a_1,\Omega)=0$ when
$|\Omega|=R^*$. In what follows, we assume that
\begin{description}
\item[(F2)]$a_1<d_1$.
\end{description}

\begin{theorem}
Assume that $\frac{a_1}{a_2}>\max\left\{\frac{b_1}{b_2},~\frac{c_1}{c_2}\right\}$. If $h_\infty-g_\infty<\infty$, then $h_\infty-g_\infty\le R^*$.
\label{Thm331}
\end{theorem}

\begin{proof}
We first prove that $h_\infty-g_\infty\le R^*$. Otherwise $h_\infty-g_\infty>R^*$ and there exists $T_1>0$ such
that $h(t)>h_\infty-\epsilon_2$, $g(t)<g_\infty+\epsilon_2$
and $h(t)-g(t)>h_\infty-g_\infty-2\epsilon_2>R^*$ for all
$t\ge T_1$ and some small $\epsilon_2$ with $0<\epsilon_2<\frac 12\left(\frac {a_1}{c_1}-\frac{a_2}{c_2}\right)$.

Since $\limsup_{t\rightarrow+\infty}v(t,x)\le\frac{a_2}{c_2}$, for the above
$\epsilon_2$ we can find $T_2\ge T_1$ such that $v(t,x)\le\frac{a_2}{c_2}+\epsilon_2:=A_2$ for all $t\ge T_2$ and $x\in\mathbb{R}$.

Denote $\Omega_\infty^{\epsilon_2}=(g_\infty+\epsilon_2,h_\infty-\epsilon_2)$, and consider the following problem
\begin{equation}
\left\{
\begin{aligned}
&w_t-d_1\left[\int_{\Omega_\infty^{\epsilon_2}}J(x-y)w(t,y)dy-w\right]
=w\left[a_1-c_1A_2-b_1w\right],& &t\ge T_2,~x\in\Omega_\infty^{\epsilon_2},\\
&w(t,x)=0,& &t\ge T_2,~x\not\in\Omega_\infty^{\epsilon_2},\\
&w(T_2,x)=u(T_2,x),& &x\in\Omega_\infty^{\epsilon_2}.
\end{aligned}
\right.
\label{331}
\end{equation}
It is well-known that (see\cite{KaoCHYDCDS2010,Hutson2003JMB})
problem (\ref{331}) admits a unique positive solution denoted
by $\underline w(t,x)=\underline w_{\epsilon_2}(t,x)$. It then
follows from the comparison principle that
$$
u(t,x)\ge\underline w(t,x)~\text{ for }~t>T_2~\text{ and }~x\in[g_\infty+\epsilon_2,h_\infty-\epsilon_2].
$$
In addition, if we use $\lambda_1^{\epsilon_2}(\infty)$ to denote the
principal eigenvalue of problem (\ref{331}), then $\lambda_1
^{\epsilon_2}(\infty)<\lambda_1(R^*)=0$. It then follows
from Theorem \ref{Thm313} (see also Hutson et al.\cite[Theorem 3.6]{Hutson2003JMB}) that
$$
\underline w(t,x)\rightarrow\frac{a_1}{b_1}-\frac{c_1}{b_1}
A_2~\text{ in }~C([g_\infty+
\epsilon_2,h_\infty-\epsilon_2])~\text{ as }~t\rightarrow+\infty.
$$
It turns out that $\liminf_{t\rightarrow+\infty}u(t,x)\ge\frac{a_1}{b_1}-\frac{c_1}{b_1}A_2>0$ uniformly in $[g_\infty+\epsilon_2,h_\infty-\epsilon_2]$.

Similarly, the following problem
\begin{equation}
\left\{
\begin{aligned}
&w_t-d_1(J*w-w)=w\left[a_1-c_1A_2-b_1w\right],& &t\ge T_2,~x\in(g_\infty,h_\infty),\\
&w(t,x)=0,& &t\ge T_2,~x\not\in(g_\infty,h_\infty),\\
&w(T_2,x)=\tilde u(T_2,x),& &x\in(g_\infty,h_\infty).
\end{aligned}
\right.
\label{332}
\end{equation}
with $\tilde u(T_2,x)=u(T_2,x)$ for $x\in[g(T_2),h(T_2)]$ and
$\tilde u(T_2,x)=0$ if $x\in(g_\infty,g(T_2))\cup(h(T_2),h_\infty)$
admits a unique positive solution $\bar w(t,x)$ such that
$u(t,x)\le\bar w(t,x)$ for $t>T_2$ and $x\in[g(t),h(t)]$,
and
$$
\bar w(t,x)\rightarrow\frac{a_1}{b_1}-\frac{c_1}{b_1}A_2~\text{ in }~C([g_\infty,h_\infty])~\text{ as }~t\rightarrow+\infty.
$$
Thus, there holds $\limsup_{t\rightarrow+\infty}u(t,x)\le\frac{a_1}{b_1}-\frac{c_1}{b_1}A_2$
for $x\in[g_\infty,h_\infty]$. By taking $\epsilon_2\rightarrow0$ deduces that
$\lim_{t\rightarrow+\infty}
u(t,x)=\frac{a_2}{b_1}\left(\frac{a_1}{a_2}-\frac{c_1}{c_2}\right)>0$
for $x\in[g_\infty,h_\infty]$. Combining this with Theorem \ref{uv-vanishing}
immediately deduces that $h_\infty-g_\infty=\infty$, this contradiction proves that $h_\infty-g_\infty<R^*$.
\end{proof}

Theorem \ref{Thm331} also implies that if $2h_0\ge R^*$,
then $h_\infty-g_\infty=\infty$.

\begin{theorem}
Assume that $\frac{a_1}{a_2}>\max\left\{\frac{b_1}{b_2},~\frac{c_1}{c_2}\right\}$ holds. Let $(u,v,g,h)$ be the
unique positive solution of (\ref{101}) with $h_\infty
-g_\infty=\infty$, then $\lim_{t\rightarrow+\infty}(u,v)(t,x)
=\left(\frac{a_1}{b_1},0\right)$ uniformly in any bounded
subset of $\mathbb{R}$.
\label{Thm332}
\end{theorem}
\begin{proof}
The method here is similar as that in Theorem \ref{Thm325}. Since
$\limsup_{t\rightarrow+\infty}v(t,x)\le\frac{a_2}{c_2}$, then we can find $T_4>0$ large such that $v(t,x)\le\frac{a_2}{c_2}+\epsilon_3$ for $t\ge T_4$, where $0<\epsilon_3\ll\frac 12\left(\frac{a_1}{c_1}-\frac{a_2}{c_2}\right)$. Meanwhile,
by $h_\infty-g_\infty=\infty$, we can find $T_5$ with $T_5\ge T_4$ such that $h(t)-g(t)>R^*$ for $t\ge T_5$.

Denoting $B_1=\frac{a_2}{c_2}+\epsilon_3$, using $(\underline u,\underline g,\underline h)$ to denote
the positive solution of the following problem
\begin{equation}
\left\{
\begin{aligned}
&\underline u_t-d_1\left[\int_{\underline g(t)}^{\underline h(t)}J(x-y)\underline u(t,y)dy-\underline u\right]=\underline u\left[a_1-c_1B_1-b_1\underline u\right],
& &t\ge T_5,~x\in(\underline g(t),\underline h(t)),\\
&\underline u(t,x)=0,& &t\ge T_5,~x\not\in(\underline g(t),\underline h(t)),\\
&\underline h'(t)=\mu\int_{\underline g(t)}^{\underline h(t)}
\int_{\underline h(t)}^{+\infty}J(x-y)\underline u(t,x)dydx,& &t\ge T_5,\\
&\underline g'(t)=-\mu\int_{\underline g(t)}^{\underline h(t)}
\int_{-\infty}^{\underline g(t)}J(x-y)\underline u(t,x)dydx,& &t\ge T_5,\\
&\underline u(T_5,x)=u(T_5,x),~\underline h(T_5)=h(T_5),~\underline g(T_5)=g(T_5),& &x\in(\underline g(T_5),\underline h(T_5)).
\end{aligned}
\right.
\label{335}
\end{equation}
Then $u(t,x)\ge\underline u(t,x)$, $g(t)\le\underline g(t)$ and
$h(t)\ge\underline h(t)$ for $t\ge T_5$ and $x\in[\underline g(t),
\underline h(t)]$ by the comparison principle.

In addition, we have $\underline h(T_5)-\underline g(T_5)>R^*$. Then for $t\ge T_5$, the principle eigenvalue $\lambda_1^*=\lambda_1^*(d_1,a_1-c_1B_1,\tilde\Omega(t))$ of (\ref{335}) satisfies $\lambda_1^*\le
\lambda_1^*(d_1,a_1-c_1B_1,\tilde\Omega(T_5))<0$, where
$\tilde\Omega(t)=(\underline g(t),\underline h(t))$.

Since $\lambda_1^*<0$, then it follows from Theorem \ref{Thm313}
that
$$
\lim_{t\rightarrow+\infty}\underline u(t,x)=\frac{a_2}{b_1}
\left(\frac{a_1}{a_2}-\frac{c_1}{c_2}-\frac{c_1}{a_2}\epsilon_3
\right):=\tilde B>0.
$$
Hence $\liminf_{t\rightarrow+\infty}u(t,x)\ge\tilde B$. Then we can find $T_l$ with $T_l\ge T_5$ and $l$ with $l\ge R^*$ such that $u(t,x)\ge\tilde B$ in $[T_l,\infty)\times[-l,l]$. Therefore, for our choices of
$T_l$ and $l$, we arrive at
$$
u\ge\underline u~\text{ and }~v\le\overline v~\text{ for }~t\ge T_l~\text{ and }~x\in[-l,l],
$$
where $(\underline u,\overline v)$ denote the solution of the
following problem
\begin{equation}
\left\{
\begin{aligned}
&\underline u_t(t,x)-d_1\left[\int_{-l}^lJ(x-y)\underline u(t,y)dy
-\underline u\right]
=\underline u(a_1-b_1\underline u-c_1\overline v),& &t\ge T_l,~x\in(-l,l),\\
&\overline v_t(t,x)-d_2\left[\int_{-l}^lJ(x-y)\overline v(t,y)dy-\overline v\right]
=\overline v(a_2-b_2\underline u-c_2\overline v),& &t\ge T_l,~x\in(-l,l),\\
&\underline u(T_l,x)=\underline u_{T_l}(x)=\tilde B,~\overline v
(T_l,x)=\overline v_{T_l}(x)=\frac{a_2}{c_2}+\epsilon_3, & &x\in(-l,l),\\
&\underline u(t,\pm l)=\tilde B,~\overline v(t,\pm l)=\frac{a_2}{c_2}
+\epsilon_3,& &t\ge T_l
\end{aligned}
\right.
\label{337}
\end{equation}
with $\left(\tilde B,\frac{a_2}{c_2}+\epsilon_3\right)$ a pair of lower solution. In view of the dependence of solutions on initial data, we denote $(\underline u(t,x;\underline u_{T_l},\overline v_{T_l}),\overline v(t,x;\underline u_{T_l},
\overline v_{T_l}))$ by the solution of problem (\ref{337}). Note that $f_1^{**}:=\underline u(a_1-b_1\underline u-c_1\overline v)$ is nonincreasing in $\overline v$ and $f_2^{**}:=\overline v(a_2-b_2\underline u-c_2
\overline v)$ is nonincreasing in $\underline u$, then (\ref{337}) generates a monotone dynamical system with respect to the order $\le_2$. This implies that for $t_2>t_1\ge T_l$ and $x\in[-l,l]$,
$$
\left(\tilde B,\frac{a_2}{c_2}+\epsilon_3\right)\le_2\left(
\underline u(t_2,x;\underline u_{T_l},\overline v_{T_l}),
\overline v(t_2,x;\underline u_{T_l},\overline v_{T_l})\right)
\le_2\left(\underline u(t_1,x;\underline u_{T_l},\overline v_{T_l}),
\overline v(t_1,x;\underline u_{T_l},\overline v_{T_l})\right).
$$
Hence
$\lim_{t\rightarrow+\infty}\left((\underline u(t,x;
\underline u_{T_l},\overline v_{T_l}),\overline v(t,x;
\underline u_{T_l},\overline v_{T_l}))\right)=\left(\overline u_l(x),\underline v_l(x)\right)$ uniformly in $[-l,l]$, where $\left(\underline u_l(x),\overline v_l(x)\right)$ satisfies
\begin{equation}
\left\{
\begin{aligned}
&-d_1\left[\int_{-l}^lJ(x-y)\underline u_l(y)dy-\underline u_l(x)\right]=\underline u_l(a_1-b_1\underline u_l-c_1\overline v_l),& &-l<x<l,\\
&-d_2\left[\int_{-l}^lJ(x-y)\overline v_l(y)dy-\overline v_l(x)\right]=\overline v_l
(a_2-b_2\underline u_l-c_2\overline v_l),& &-l<x<l,\\
&\underline u_l(-l)=\underline u_l(l)=\tilde B,
~\overline v_l(l)=\overline v_l(-l)=\frac{a_2}{c_2}+\epsilon_3
\end{aligned}
\right.
\label{338}
\end{equation}
and $\lim_{l\rightarrow\infty}\left(\underline u_l(x),\overline v_l(x)\right)=\left(
\underline u^*(x),\overline v^*(x)\right)$ with $\left(\underline u^*(x),\overline v^*(x)\right)$ satisfies
\begin{equation}
\left\{
\begin{aligned}
&-d_1(J*\underline u^*-\underline u^*)=\underline u^*(a_1
-b_1\underline u^*-c_1\overline v^*),& &x\in\mathbb{R},\\
&-d_2(J*\overline v^*-\overline v^*)=\overline v^*
(a_2-b_2\underline u^*-c_2\overline v^*),& &x\in\mathbb{R},\\
&\underline u^*(x)\ge\tilde B,~\overline v^*(l)\le\frac{a_2}{c_2}
+\epsilon_3,& &x\in\mathbb{R}.
\end{aligned}
\right.
\label{339}
\end{equation}

Moreover, since $\frac{a_1}{a_2}>\max\left\{\frac{b_1}{b_2},
~\frac{c_1}{c_2}\right\}$, then by Morita et al.\cite{MoritaSIAM2009}
that the solution $(u_2,v_2)$ of the following
\begin{equation}
\left\{
\begin{aligned}
&(u_2)_t=u_2(a_1-b_1u_2-c_1v_2),& &t>0,\\
&(v_2)_t=v_2(a_2-b_2u_2-c_2v_2),& &t>0,\\
&u_2(0)=\tilde B,~v_2(0)=\frac{a_2}{c_2}+\epsilon_3
\end{aligned}
\right.
\label{340}
\end{equation}
satisfies $\lim_{t\rightarrow+\infty}(u_2(t),v_2(t))=\left(\frac{a_1}{b_1},
0\right)$, which implies that the solution $(U^*(t,x),V^*(t,x))$ of
\begin{equation}
\left\{
\begin{aligned}
&U^*_t-d_1(J*U^*-U^*)=U^*(a_1-b_1U^*-c_1V^*),& &t>0,~x\in\mathbb{R},\\
&V^*_t-d_2(J*V^*-V^*)=V^*(a_2-b_2U^*-c_2V^*),& &t>0,~x\in\mathbb{R},\\
&U^*(0,x)=\tilde B,~V^*(0,x)=\frac{a_2}{c_2}+\epsilon_3,& &x\in\mathbb{R}
\end{aligned}
\right.
\label{341}
\end{equation}
satisfies $\lim_{t\rightarrow+\infty}\left(U^*(t,x),V^*(t,x)
\right)=\left(\frac{a_1}{b_1},0\right)$ uniformly in any
bounded subset of $\mathbb{R}$.

By using the comparison principle to problems (\ref{339}) and (\ref{341})
we obtain that $\underline u^*(x)\ge\frac{a_1}{b_1}$ and
$\overline v^*(x)\le V^*(t,x)$ for all $x\in\mathbb{R}$, which indicates that
$\underline u^*(x)\ge\frac{a_1}{b_1}$ and $\underline v^*(x)=0$ for all $x\in\mathbb{R}$, and then $\underline u_l(x)\ge\frac{a_1}{b_1}$ and
$\overline v_l(x)=0$ for $x\in(-l,l)$, and hence $\underline u(t,x;\underline u_{T_l},\overline v_{T_l})\ge\frac{a_1}{b_1}$ and
$\overline v(t,x;\underline u_{T_l},\overline v_{T_l})=0$ as
$t\rightarrow+\infty$. Further, we get $u(t,x)\ge\frac{a_1}{b_1}$ and $v(t,x)=0$ as $t\rightarrow+\infty$. Recall the proof of Theorem \ref{Thm325} that
$\limsup_{t\rightarrow+\infty}u(t,x)\le\frac{a_1}{b_1}$ uniformly for $x\in\mathbb{R}$. And then $\lim_{t\rightarrow+\infty}
\|u(t,\cdot)\|_{C([g(t),h(t)])}=\frac{a_1}{b_1}$ and $\lim_{t\rightarrow+\infty}\|v(t,\cdot)\|_{L^\infty(\mathbb{R})}=0$.
This completes the proof.
\end{proof}

From now on, we can establish the spreading-vanishing
dichotomy for problem (\ref{101}).
\begin{theorem}
Assume that $\frac{a_1}{a_2}>\max\left\{\frac{b_1}{b_2},~\frac{c_1}{c_2}\right\}$ holds. Let $(u,v,g,h)$ be the unique positive solution of (\ref{101}) with $v_0(x)\not\equiv0$. Then the following alternatives holds: Either

(i) \underline{spreading of $u$}: $h_\infty-g_\infty=\infty$ and
$\lim_{t\rightarrow+\infty}(u,v)=\left(\frac
{a_1}{b_1},0\right)$ uniformly in any bounded subset
of $\mathbb{R}$; or

ii) \underline{vanishing of $u$}: $h_\infty-g_\infty\le R^*$ and $\lim_{t\rightarrow+\infty}\left(u,v\right)=\left(0,\frac{a_2}{c_2}\right)$ uniformly in any bounded subset of $\mathbb{R}$.
\label{Thm333}
\end{theorem}

We obtained that spreading will always happens as long
as $2h_0\ge R^*$. Following is a criteria that
devoted to the expanding ability $\mu$ to govern the
spreading alternative if $2h_0<R^*$.

\begin{theorem}
Assume that $\frac{a_1}{a_2}>\max\left\{\frac{b_1}{b_2},~\frac{c_1}{c_2}\right\}$ holds. If $2h_0<R^*$, then there exists $\underline\mu>0$ such
that $h_\infty-g_\infty=\infty$ if $\mu\ge\underline\mu$.
\label{spreading-case}
\end{theorem}

\begin{proof}
Suppose on the contrary that $h_\infty-g_\infty<\infty$
for all $\mu>0$ if $2h_0<R^*$. It then follows from Theorem
\ref{Thm331} that $h_\infty-g_\infty\le R^*$. And hence we can find a large $T$ such that $2h_0<h(T)-g(T)\le R^*$.

The free boundary conditions $h'(t)$ and $g'(t)$ yields that
\begin{align*}
&h(t)=h(T)+\mu\int_T^t\int_{g(\tau)}^{h(\tau)}\int_{h(\tau)}^{+\infty}J(x-y)
u(\tau,x)dydxd\tau,\\
&g(t)=g(T)-\mu\int_T^t\int_{g(\tau)}^{h(\tau)}\int_{-\infty}^{g(\tau)}J(x-y)
u(\tau,x)dydxd\tau.
\end{align*}
By letting $t\rightarrow+\infty$ deduce that
\begin{align*}
R^*-2h_0>&~\left(h_\infty-g_\infty\right)-\left(h(T)-g(T)\right)\\
=&~\mu\int_T^\infty\left(\int_{g(\tau)}^{h(\tau)}\int_{h(\tau)}^{+\infty}+\int_{g(\tau)}^{h(\tau)}\int_{-\infty}^{g(\tau)}\right)J(x-y)u(\tau,x)dydxd\tau.
\end{align*}

Moreover, it follows from assumption
$(A1)$ that there exist constants $\epsilon_0>0$, $\delta_0>0$ such that $J(x-y)\ge\delta_0$ if $|x-y|\le\epsilon_0$. Therefore,
\begin{align*}
&~\mu\int_{g(\tau)}^{h(\tau)}\int_{h(\tau)}^{+\infty}J(x-y)u(\tau,x)dydx\\
\ge~&\mu\int_{h(\tau)-\frac{\epsilon_0}{2}}^{h(\tau)}\int_{h(\tau)}^{h(\tau)+
\frac{\epsilon_0}{2}}J(x-y)u(\tau,x)dydx
\ge\frac{1}{2}\mu\epsilon_0\delta_0\int_{h(\tau)-\frac{\epsilon_0}{2}}^{h(\tau)}u(\tau,x)dx
\end{align*}
and
$$
\mu\int_{g(\tau)}^{h(\tau)}\int_{-\infty}^{g(\tau)}J(x-y)u(\tau,x)dydx\ge\frac{1}{2}\mu\epsilon_0\delta_0\int_{g(\tau)}^{g(\tau)+\frac{\epsilon_0}{2}}u(\tau,x)dx.
$$
Therefore, we have
\begin{align*}
R^*-2h_0>~&\mu\int_T^\infty\left(\int_{g(\tau)}^{h(\tau)}\int_{h(\tau)}^{+\infty}+\int_{g(\tau)}^{h(\tau)}\int_{-\infty}^{g(\tau)}\right)J(x-y)u(\tau,x)dydxd\tau\\
\ge~&\frac{1}{2}\mu\epsilon_0\delta_0\int_T^{\hat T}\left(\int_{g(\tau)}^{g(\tau)+\frac{\epsilon_0}{2}}+\int_{h(\tau)-\frac{\epsilon_0}{2}}^{h(\tau)}\right)u(\tau,x)dxd\tau,
\end{align*}
where $\hat T\in(T,\infty)$. Since $u(t,x)>0$, then $\Delta(\delta_0,
\epsilon_0,T,\hat T)$ defined by
$$
\Delta(\delta_0,\epsilon_0,T,\hat T)=\frac{1}{2}\epsilon_0\delta_0\int_T^{\hat T}\left(\int_{g(\tau)}^{g(\tau)+\frac{\epsilon_0}{2}}+\int_{h(\tau)-\frac{\epsilon_0}{2}}^{h(\tau)}\right)u(\tau,x)dxd\tau>0
$$
and thus
$$
0<\mu<\left(R^*-2h_0\right)\left\{\Delta(\delta_0,\epsilon_0,T,\hat T)
\right\}^{-1}:=\underline\mu,
$$
which in turn indicates that we can find $\underline\mu>0$
such that for all $\mu\ge\underline\mu$, there holds $h_\infty-g_\infty=\infty$ even though $2h_0<h^*$.
\end{proof}

Theorem \ref{spreading-case} states that the superior competitor $u$ will spread eventually if the expanding ability $\mu\ge\underline\mu>0$ even though the initial occupied stage is small. Below is a criteria on $\mu$ that govern the vanishing case.

\begin{theorem}
Assume that $\frac{a_1}{a_2}>\max\left\{\frac{b_1}{b_2},~\frac{c_1}{c_2}\right\}$ holds. If $2h_0<R^*$, then there
exists $\overline\mu\ge0$ such that $h_\infty-g_\infty<\infty$
if $0<\mu\le\overline\mu$.
\label{vanishing-case}
\end{theorem}
\begin{proof}
It is easy to find that $u(t,x)$ in problem (\ref{101}) satisfies the following
\begin{equation*}
\left\{
\begin{aligned}
&u_t\le d_1\left[\int_{g(t)}^{h(t)}J(x-y)u(t,y)dy-u\right]+u(a_1-b_1u),
& &t>0, x\in(g(t),h(t)),\\
&u(t,g(t))=u(t,h(t))=0,& &t>0,\\
&h'(t)=\mu\int_{g(t)}^{h(t)}\int_{h(t)}^{+\infty}J(x-y)u(t,x)dydx,& &t>0,\\
&g'(t)=-\mu\int_{g(t)}^{h(t)}\int_{-\infty}^{g(t)}J(x-y)u(t,x)dydx,& &t>0,\\
&u(0,x)=u_0(x),~h(0)=-g(0)=h_0,& &x\in[-h_0,h_0],
\end{aligned}
\right.
\end{equation*}
which immediately deduces that $(u,g,h)$ is a lower solution of
\begin{equation}
\left\{
\begin{aligned}
&\hat u_t=d_1\left[\int_{\hat g(t)}^{\hat h(t)}J(x-y)\hat u(t,y)dy-\hat u\right]
+\hat u(a_1-b_1\hat u), & &t>0, x\in(\hat g(t)),\hat h(t)),\\
&\hat u(t,\hat g(t))=\hat u(t,\hat h(t))=0,& &t>0,\\
&\hat h'(t)=\mu\int_{\hat g(t)}^{\hat h(t)}\int_{\hat h(t)}^{+\infty}J(x-y)
\hat u(t,x)dydx,& &t>0,\\
&\hat g'(t)=-\mu\int_{\hat g(t)}^{\hat h(t)}\int_{-\infty}^{\hat g(t)}J(x-y)
\hat u(t,x)dydx,& &t>0,\\
&\hat u(0,x)=u_0(x),~\hat h(0)=-\hat g(0)=h_0,& &x\in[-h_0,h_0].
\end{aligned}
\right.
\label{upper-bound-equation-1}
\end{equation}
Note that problem (\ref{upper-bound-equation-1}) is the model that studied in
\cite{CDLL-2018}, and it follows from \cite[Theorem 3.12]{CDLL-2018} that
there exists $\overline\mu\ge0$ such that vanishing of $\hat u$ happens if $0<\mu\le\overline\mu$ since $2h_0<R^*$. Therefore, we obtain that
$h(t)-g(t)\le\hat h(t)-\hat g(t)<\infty$ and $u\le\hat u\rightarrow0$ as $t\rightarrow\infty$ if $0<\mu\le\overline\mu$. This completes the proof.
\end{proof}

\begin{theorem}
Assume that $\frac{a_1}{a_2}>\max\left\{\frac{b_1}{b_2},~\frac{c_1}{c_2}\right\}$ and $2h_0<R^*$. Then there exists $\mu^*\ge0$ such that
$h_\infty-g_\infty=\infty$ if $0<\mu\le\mu^*$ and $h_\infty-g_\infty<\infty$ if
$\mu>\mu^*$.
\label{Thm335}
\end{theorem}

\begin{proof}
See \cite[Theorem 3.14]{CDLL-2018} for the detailed proof.
\end{proof}

\section*{Acknowledgments}
\noindent

Our sincere thanks goes to Professor Yihong Du (University of New England) for proposing this question. Research of  Wan-Tong Li was partially supported by NSF of China (11731005, 11671180). Research of Jie Wang was partially supported by NSF of China (11701243).

\end{document}